\documentclass{article}
\usepackage[utf8]{inputenc}
\newcounter{defn}
\newcounter{rmk}

\usepackage{amsmath}
\usepackage{float}

\usepackage{comment}
\usepackage[linktocpage=true]{hyperref}
\usepackage[margin=1.2 in, top=1 in, bottom= 1.2 in]{geometry}

\usepackage[english]{babel}
\usepackage{amsfonts}
\usepackage{mathrsfs}
\usepackage[mathscr]{euscript}
\usepackage{latexsym}
\usepackage{amssymb}
\usepackage{mathtools}
\usepackage{graphicx}
\usepackage{tikz}
\usetikzlibrary{calc}
\usepackage{placeins}
\usepackage{cancel}
\usepackage{pgfplots}
\pgfplotsset{compat=1.18}

\usepackage{enumitem}
\setlist{nolistsep}
\usepackage{amsthm}
\usepackage{relsize}
\usepackage{nicefrac}
\usepackage[capitalize]{cleveref}
\usepackage{longtable}
\usepackage{booktabs}

\newtheoremstyle{plain}{3mm}{3mm}{\slshape}{}{\bfseries}{.}{.5em}{}
\newtheoremstyle{definition}{2mm}{2mm}{}{}{\bfseries}{.}{.5em}{}
\theoremstyle{plain}

\theoremstyle{plain}
\newtheorem*{namedthm}{\namedthmname}
\newcounter{namedthm}
	

\newcommand{\D}{\mathbb{D}}

\newcommand{\Z}{\mathbb{Z}}
\newcommand{\R}{\mathbb{R}}

\newcommand{\E}{\mathcal{E}}

\newcommand{\eps}{\epsilon}

\newcommand{\thee}{\Theta(U)}
\newcommand{\en}{\beta_1}
\newcommand{\mb}[1]{\mathbf{#1}}

\newtheorem{theorem}{Theorem}[section]
\newtheorem{lemma}{Lemma}

\crefname{lemma}{lemma}{lemmas}
\Crefname{lemma}{Lemma}{Lemmas}

\theoremstyle{definition}
\newtheorem{definition}[defn]{Definition}
\newtheorem{remark}[rmk]{Remark}
\theoremstyle{plain}

\newcommand{\proj}{\nabla_{A_{m}}\cdot \hat{v_1} }

\title{On the dimension drop for harmonic measure on uniformly non-flat Ahlfors-David regular boundaries. }
\author{ Aritro Pathak}
\date{}

\begin{document}

\maketitle

\begin{abstract}
 We extend earlier results of Azzam on the dimension drop of the harmonic measure for a domain $\Omega\subset \R^{n}$ with $n\geq 3$, with Ahlfors-David regular boundary $\partial\Omega$ of dimension $s$ with $n-1-\delta_0 \leq s\leq n-1$, that is uniformly non flat. Here $\delta_0$ is a small positive constant dependent on the parameters of the problem. Our novel construction relies on elementary geometric and potential theoretic considerations. We avoid the use of Riesz transforms and compactness arguments, and also give quantitative bounds on the $\delta_0$ parameter.
\end{abstract}

\section{Introduction.}
 The study of the harmonic measure in relation to the geometric properties of the boundary of a domain $\Omega\subset \R^n$, is a classical subject in analysis. The harmonic measure arises in an important way in the Dirichlet problem. Previous studies, going back to the work in \cite{Rie23}, study the mutual absolute continuity of the harmonic measure with respect to the surface measure in various settings. 

  More recently, there has been a focus on the harmonic measure for boundaries $\partial\Omega$ which have codimension larger than $1$, through the conjectures in \cite{Vol22}, and the recent work of \cite{DJJ23, Tol24} for the special situation of boundaries that are a subset of a codimension one hyperplane or a manifold. In \cite{Tol24}, it is shown that in $\R^{n}$, for such a boundary lying on a codimension one hyperplane, with Hausdorff dimension greater than or equal to $n-3/2-\epsilon$ for some small enough $\eps$, there is a dimension drop of the harmonic measure. In \cite{DJJ23}, it is shown that in $\R^2$, a boundary with a small enough Hausdorff dimension (less than $0.2497$) in $\R^{3}$, and lying on a line, the harmonic measure is actually absolutely continuous with the surface measure.  
  Prior work by \cite{Car85, JW88, Vol93, Vol92, UZ02, Bat06, Bat00, Bat96} focused on the case of dynamically defined boundaries which satisfy some kind of totally disconnected properties, for which the dimension drop phenomenon was established. This enabled the use of some ergodic theory. Extensions of this work to the case of boundaries that are not dynamically defined or totally disconnected, but satisfying a uniformly non flatness condition was undertaken in \cite{Az20}, for the case of boundaries with codimension less than $1$. Here we work on extending these results.

Here we give a new argument, that ostensibly avoids the use of Riesz transforms, and the techniques stemming from the work of \cite{AHMMMTV16, Tol15}. We employ a double integral averaging argument with a Riesz type kernel, in \cref{lemma10}, and use in \cref{lemma12} a change of pole argument for the harmonic measure, to localize the problem with poles of harmonic measures close to a specific surface cube under question, by closely following the corresponding argument of Lemma 4.3 of \cite{Az20}. We also give explicit bounds on the dimension of the boundary, in terms of the parameters of the problem, for which the dimension drop of the harmonic measure is admissible.

  For prior work on related questions on the harmonic measure, and the dimension drop phenomenon, as well as related questions in uniform rectifiability of the boundary and absolute continuity of the harmonic measure, we refer the reader to, \cite{Wol95, Tol15, Pop98, MV86, Mey09, Mattila, Mak85, LVV05, KT99, HMMTZ17, HKM, Harmonic-Measure, Bou87, BE17, AAM16, AHMMMTV16, AMT17, AM15}. 

Let $\omega=\omega_\Omega$ denote the harmonic measure for a given domain $\Omega$. In this paper we, prove the following result.

\begin{theorem}\label{mainthm}
Given any integer $n \geq 3$, $C_{1} > 1$, and $0<\beta <1$, the following holds. Define, $N_\beta =\lceil \frac{1}{\beta} \rceil$ and $\beta_1:= \frac{1}{4N_{\beta}+2}$.

 Suppose $$\delta_0=\delta_0(\beta,n, C_1):=\text{min}\Big(\frac{1}{2} \frac{1}{c(C_1,c_0)}\beta_1^{\frac{3nD}{\beta_1^n}}\Big).$$

 Suppose $\Omega \subseteq \mathbb{R}^{n}$ is a connected
domain. Further, for any point $x\in \R^n$ and any $r< \text{diam}(\Omega)$, we have,
\begin{equation}\label{bilateral}
b\beta_{\omega}(x,r)
:= \inf_{V}\left[
\sup_{y \in B(x,r)\cap \partial \Omega} \frac{\operatorname{dist}(y,V)}{r}
\;+\;
\sup_{y\in V\cap B(x,r)} \frac{\operatorname{dist}(y,\partial\Omega)}{r}
\right]
\ge \beta > 0,
\end{equation}
where the infimum is over all $(n-1)$-dimensional planes $V \subseteq \mathbb{R}^{n}$. 

Further, assume that the boundary $\partial\Omega$ is Ahlfors-David $s$-regular, ($s,C_1$ AD-regular) meaning if
$\sigma = \mathcal{H}^{s}\!\!\mid_{\partial\Omega}$, then
\begin{equation}\label{adr}
C_{1}^{-1} r^{s} \le \sigma(B(x,r)) \le C_{1} r^{s}
\quad \text{for all } x \in \partial\Omega,\; 0 < r < \operatorname{diam}(\partial\Omega).
\end{equation}

Further, assume that in the case of $s=n-1$, the domain $\Omega$ satisfies the interior corkscrew condition.


Then for any $n-1-\delta_0 \leq  s \leq n-1$, we have, $\dim \omega_\Omega <  s$, meaning there is a set $K$ with
$\dim K <  s$ such that $\omega_\Omega(K^{c}) = 0$.
\end{theorem}

Note that, $\beta_1\leq \beta/4$. In particular, we will only employ the \cref{bilateral} condition for any point $x\in\R^n$, particularly for planes passing through $x$. We do not restrict this condition to only points $x\in\partial\Omega$. Of course, if we were to choose the point $x$ and the plane $V\ni x$ far from the boundary $\partial\Omega$ depending on $\beta$, then the condition is vacuously true, and so the statement is only interesting when the point $x$ and the chosen plane $V$ are close to the boundary.

This condition roughly states that at any given scale the boundary cannot be well approximated by a plane; if all points of the boundary lie very close to a plane on a given scale locally, then locally the plane must have enough `holes' which are not close to the boundary, and if there is a plane so that all points of the plane locally lie very close to the boundary, then locally there must also be points of the boundary at a controlled distance away from this particular plane. The coefficient $b\beta_{\omega}(x,r)$ is natural in the context of many questions in rectifiability, and a weaker assumption than just requiring that points of the boundary locally stay away from planes by a uniform factor $\beta$, in which case we would omit the second term on the right, in the expression for $b\beta_{\omega}(x,r)$ in \cref{bilateral}.

\begin{remark}
    Note that we give an explicit bound on $\delta_0$ parameter, in terms of the parameter $\beta$ that characterizes the uniform non-flatness, as well as the Ahlfors-David regularity parameters. We do not employ the notion of the $\alpha$ numbers as in \cite{Tol15}, nor use a compactness argument as in \cite{Az20}, nor use the co-dimension one characterization of rectifiability and the absolute continuity of harmonic measure stemming from the work of \cite{AHMMMTV16}.
\end{remark}

\begin{remark}Note that the case of $(n-1)< s< n$ is proved more directly in \cite{Az20} by combining Lemma 4.1 and 4.3 of their paper. We omit the proof of this case in our paper, which is more direct with the use of the one sided Caffarelli-Fabes-Mortola-Salsa type estimate and a touching point argument. Instead, we focus on the larger codimension case.
\end{remark}

\begin{remark}
    In general, the Cantor type sets considered in \cite{Bat96} are not such $s$- dimensional Ahlfors regular sets. Given a specific example of Cantor sets with varying dissection ratios in different stages of the construction of those Cantor sets, the arguments of this paper can be modified appropriately for such a situation. Specifically, if the Ahlfors regularity exponent varies in a bounded way, one can adopt the methods of the proof of the case of $s=n-1$, $s\leq n-1$ from this paper, along with the more direct argument of $s>n-1$ from the earlier results.
\end{remark}
\begin{remark}
    We have not carried out the calculations in the plane $\R^2$, where with some modifications to the argument and of the fundamental solution which becomes logarithmic, one would expect a result of a similar structure to be true.
\end{remark}

\section{Basic estimates and overview of the proof.}
The Green function is constructed by setting
\begin{equation}\label{imp2}
G(X,Y) := \mathcal{E}(X,Y)
\;-\;
\int_{\partial\Omega} \mathcal{E}(X , z)\, d\omega^{Y}(z),
\end{equation}
where $\mathcal{E}(X) := c_{n} |X|^{2-n}$ is the usual fundamental solution for the
Laplacian in $\mathbb{R}^{n}$. Note that we have chosen to take a positive sign for the contribution of the pole, and a negative sign for the contribution of the potential due to the harmonic measure. Also, $G(X,Y)$ is positive everywhere inside the domain, by construction, and it vanishes at the boundary. See for example, section 2 of \cite{HMMTZ17}.

The basic idea behind our proof is to first isolate some specific boundary ball $B$ and consider the harmonic measure due to some pole located within this ball centered on the boundary and with the pole being a controlled distance away from the boundary. This can be phrased in terms of the corkscrew point relative to this boundary ball.

We isolate `corner' regions, which are sub-balls of $B$ centered on a point on the boundary that is at a minimum distance from the given corkscrew point. If the harmonic measure in such a sub-ball centered on such a 'corner' point were to be approximately constant, one shows that at the `corner' point, there is a non zero and non uniform gradient of the potential due to this harmonic measure in this sub-ball.

Now, the entire surface must remain an equipotential, precisely at the value of $0$. This non uniform gradient due to the local harmonic measure has to be compensated by the contribution to the potential, coming from the harmonic measure located further away, along with the contribution due to the pole. In effect, one shows that the gradient due the harmonic measure further away, and the second order contribution from the harmonic measure located nearby and further away, both can be controlled in a manner that this compensation to the local non-uniform contribution cannot happen exactly. This is a contradiction, and we can thus quantify that the harmonic measure cannot be nearly constant, in subballs centered on this corner point. Given any arbitrarily large $M>1$, we find some $l(M)$ dependent on $M$ and a dyadic cube $C(M)$ of length $l(M)$, by a careful repetition of this argument, so that the average Poisson kernel over the cube $C(M)$ is at least a factor $M$ bigger than the average Poisson kernel over $B$:
\begin{align}\label{compare}
    k^{p_B}(C(M))\geq M k^{p_B}(B).
\end{align}

Next, by a Markov chain argument effectively outlined in Lemma 4.3 of \cite{Az20}, we compare the harmonic measure due the pole $p_B\in B$, that is a bounded distance away from the boundary relative to the radius of $B$, to the harmonic measure due to the pole at the a-priori chosen point $p$ which is located at an arbitrary point in the domain. This leads us to a lower bound of the same form as in \cref{compare}, but with a different constant $M_2$.

This happens in every dyadic cube in $\partial\Omega$. This eventually leads to a dimension drop of the harmonic measure. 

\section{Preliminaries}

The following lemma is originally due to Bourgain( \cite[Lemma 1]{Bou87}).

\begin{lemma} 
\label{l:bourgain}
If $\Omega\subseteq \R^{n}$ be a domain that has Ahlfors-David $s$-regular boundary, with constant $c_{1}$ and $s>n-2$, then there is $b\in (0,1)$ so that 
\begin{equation}
\label{e:bou}
\omega_{\Omega}^{x}(B(\xi,r))\geq c_{0} \;\;\mbox{ for }\xi\in \partial\Omega, \;\; 0<r<\text{diam} \partial\Omega,\;\; \mbox{ and }x\in B(\xi,b r).
\end{equation}
\end{lemma}

An alternate formulation of the above gives us the following.
\begin{lemma}\label{r:bourgain}

If $\Omega\subseteq \R^{n}$ has Ahlfors-David $s$-regular boundary, then for any point $x\in \Omega$ we have that, 
\begin{align}
    \omega_{\Omega}^{x}(B(\hat{x},10\delta(x)))\geq c_0,
\end{align}
    where $\delta(x)$ is the distance of the point $x$ to the boundary, and $\hat{x}$ is a point on the boundary $\partial\Omega$ closest to $x$.
\end{lemma}


\begin{lemma} \label{l:holder}
Let $\Omega\subsetneq \R^{n}$ be an open set that satisfies the Capacity Density condition and let $x\in \partial\Omega$. Then there is $\alpha>0$ so that for all $0<r<\text{diam}(\Omega)$,
\begin{equation}\label{e:holder}
 \omega_{\Omega}^{y}({B}(x,r)^{c})\lesssim \Big({\frac{|x-y|}{r}}\Big)^{\alpha},\quad \mbox{ for all } y\in \Omega\cap B(x,r),
 \end{equation}
where $\alpha$ and the implicit constant depend on $n$ and the CDC constant.
\end{lemma}

Also, it is well known that in the case of domains with boundaries $\partial\Omega$ of codimension greater than $1$, we automatically have the corkscrew condition satisfied.

\begin{definition}\label{def1}
    Corkscrew domain:  An open set $\Omega\subset \R^{n}$ satisfies the interior corkscrew condition if for some uniform constant $c$ with $0<c<1$, and for every surface ball $\Delta:=\Delta(x,r)$ with $x\in \partial \Omega$ and $0<r<\text{diam}\partial\Omega$, there is a ball $B(A(x,r),cr)\subset \Omega\cap B(x,r)$. The point $A(x,r)\in \Omega\cap B(x,r)$
 is called an interior corkscrew point relative to $\Delta$.
 \end{definition}


%
%
%

We also introduce the multipole expansion, which is used to estimate the strength of the derivative and the second derivative terms for the potential due to the harmonic measure on the boundary in arbitrary dimensions $n\geq 3$.
\begin{definition} 
Multipole expansion:
\begin{align}\label{mult expansion}
\frac{1}{\lvert \mathbf{r} - \mathbf{h} \rvert}
=
\frac{1}{r}
+ \frac{\mathbf{r} \cdot \mathbf{h}}{r^{3}}
+ \frac{1}{2 r^{3}}
\left(
3 \left( \frac{\mathbf{r} \cdot \mathbf{h}}{r} \right)^{2}
- |\mathbf{h}|^2
\right)
+ O(\frac{h^3}{|\mathbf{r}^4|}),
\end{align}

where we have $r=|\mathbf{r}|$.
\end{definition}

Note that for $\mathcal{E}(x,y)$ being the contribution at $x$, from the fundamental solution due to the pole in $y$, we can write, using the Taylor theorem to the second order, for $\mb{q_2}=\mb{q_1}+\mb{h}$, that, 
\begin{align}\label{taylor}
    \E(\mb{q_2},\mb{y})-\E(\mb{q_1},\mb{y})=\mb{h}\cdot \nabla\E(\mb{q_1},\mb{y})+\frac{1}{2}\frac{\partial^2\E}{\partial x^2}\Big|_{(\mb{q_1 +\theta_y h})}|\mb{h}|^2,
\end{align}
Here $y$ is the fixed pole and $\E(,\mb{y})$ is treated as a function of the first coordinate, and we write $\theta_{\textbf{y}}$ as a function of $\textbf{y}$. Here, for simplicity of notation, we assume that we have rotated the coordinates so that $\mb{q_2}-\mb{q_1}=\mb{h}$ has a coordinate representation of $(h,0,\dots,0)$, and so it suffices to consider only the second derivative with the $x$ variable, whereas other contributions to the Hessian term drops out. Further, an elementary calculation using the multipole expansion of \cref{mult expansion} gives us estimates for the second derivative itself. 

We will use the Taylor theorem up to second order in all the estimates, to find the difference in potential at two nearby points. We call the first term on the right of \cref{taylor}, the gradient contribution, and the second term on the right as the second order contribution to the difference in the potential between the points $\mb{q_1}$ and $\mb{q_2}$. We will integrate over the $\mb{y}$ variable, using the Ahlfors-David regularity property for the surface measure, to get lower and upper estimates on the gradient and the second derivative terms, and these will be used routinely throughout the proofs, along with hypothesized upper and lower bounds to the Poisson kernel in the source regions.

We also need the following important result on the existence and the properties of the dyadic cubes, and this is quoted from Chapter 3 of \cite{DS2}.

\begin{theorem}[Existence and properties of the dyadic grid:]\label{cubelemma}
 Let $E$ be a $d$-dimensional Ahlfors-David regular set in $\mathbb{R}^{n}$ with $d$ not necessarily an integer.  
It is possible to construct a family of subsets of $E$ that behave in much the same way as do the dyadic cubes in $\mathbb{R}^n$.  
More precisely, one can construct a family $\Delta_j$, $j \in \mathbb{Z}$, of measurable subsets of $E$ with the following properties:

\begin{enumerate}
\item[(1)] Each $\Delta_j$ is a partition of $E$, i.e.,
\[
E = \bigcup_{Q \in \Delta_j} Q, \qquad
Q \cap Q' = \varnothing \quad \text{whenever } Q, Q' \in \Delta_j \text{ and } Q \neq Q'.
\]

\item[(2)] If $Q \in \Delta_j$ and $Q' \in \Delta_k$ for some $k \ge j$, then either $Q \subset Q'$ or $Q \cap Q' = \varnothing$.

\item[(3)] For all $j \in \mathbb{Z}$ and all $Q \in \Delta_j$, we have
\[
C^{-1} 2^{\,j} \le \operatorname{diam} Q \le C 2^{\,j},
\]
and
\[
C^{-1} 2^{\,jd} \le |Q| \le C 2^{\,jd}.
\]

\item[(4)] For all $j \in \mathbb{Z}$ and $Q \in \Delta_j$, the cube $Q$ has a small boundary, in the sense that
\[
\bigl|\{x \in Q : \operatorname{dist}(x, E \setminus Q) \le \tau 2^{\,j}\}\bigr|
\;+\;
\bigl|\{x \in E \setminus Q : \operatorname{dist}(x, Q) \le \tau 2^{\,j}\}\bigr|
\;\le\;
C \, \tau^{1/C} 2^{\,jd},
\]
for all $0 < \tau < 1$.
\end{enumerate}

Here $C$ is a constant that depends only on $d$, $n$, and the regularity constant for $E$.  Further $|\cdot|=\mathcal{H}^d|_{E}$ is the restriction of the $d-$ dimensional Hausdorff measure to $E$.
\end{theorem}

The value of $d$ need not be an integer above.

 Further, we also have, as a consequence of (4) , the following. 

\begin{theorem}\label{center}
There is a constant $C_1 > 0$, depending only on $d$, $n$, and the regularity constant for $E$, such that we can associate to each cube $Q \in \Delta$ a “center’’ $c(Q) \in Q$ which satisfies
\begin{equation}
\label{eq:36}
\operatorname{dist}(c(Q),\, E \setminus Q) \;\ge\; C_1^{-1} \, \operatorname{diam} Q.
\end{equation}
Thus we have the surface ball $B_Q :=\Delta(c(Q), C_{1}^{-1}\text{diam} Q)\subset Q$.

\end{theorem}

Henceforth, it should be understood from the context that we mean the harmonic measure with respect to the domain $\Omega$ and we simply write $\omega$ in place of $\omega_\Omega$.  In certain situations, we will use the harmonic measure in certain subdomains, by eventually adopting the argument in Lemma 4.3 of \cite{Az20}.

Note that the condition of \cref{bilateral} is taken to hold over any ball centered over points on the boundary. One easily sees that this is equivalent to having the condition hold over cubes instead of balls. 

Assume \cref{bilateral}. In this case, given any arbitrary ball $B(x,r)$ with $x\in \partial\Omega$, consider the cube $C(x,r)$ of minimal volume that contains $B(x,r)$. Here, $C(x,r)$ is the Euclidean cube given by $\{(y_1,\dots,y_n):x_i -r\leq y_i\leq x_i +r , 1\leq i\leq n\}$, where $x=\{x_1,\dots,x_n\}$.

Then obviously we have,
\begin{align}
    \sup_{y \in C(x,r)\cap \partial \Omega} \frac{\operatorname{dist}(y,V)}{r} \geq  \sup_{y \in B(x,r)\cap \partial \Omega} \frac{\operatorname{dist}(y,V)}{r}, \sup_{y \in C(x,r)\cap V} \frac{\operatorname{dist}(y,\partial\Omega)}{r} \geq  \sup_{y \in B(x,r)\cap V} \frac{\operatorname{dist}(y,\partial\Omega)}{r}
\end{align}

Conversely, once we have the condition of \cref{bilateral} for cubes in place of balls, by considering the balls of minimal volume containing any given cube $C(x,\mathscr{r})$, we get the converse statement. Here, $C(x,\mathscr{r})$ is the cube with center $x$ and diameter $r$.

By a similar argument, the condition of \cref{adr} can also be written in terms of cubes in place of balls. Without loss of generality, we keep the same constants for the case of the cubes, in \cref{adr}.

Thus, henceforth we use the following two conditions; for any point $x\in \R^n$ and any $r< \text{diam}(\Omega)$, we have,
\begin{equation}\label{bilateral'}\tag{1'}
b\beta_{\omega}(x,r)
:= \inf_{V}\left[
\sup_{y \in C(x,r)\cap \partial \Omega} \frac{\operatorname{dist}(y,V)}{r}
\;+\;
\sup_{y\in V\cap C(x,r)} \frac{\operatorname{dist}(y,\partial\Omega)}{r}
\right]
\ge \beta > 0,
\end{equation}
where the infimum is over all $(n-1)$-dimensional planes $V \subseteq \mathbb{R}^{n}$.
 
Further, assuming that the boundary $\partial\Omega$ is Ahlfors-David $s$-regular, ($s,C_1$ AD-regular), then we also have for some constant $C_2$,
\begin{equation}\label{adr}\tag{2'}
C_{2}^{-1} r^{s} \le \sigma(C(x,r)) \le C_{2} r^{s}
\quad \text{for all } x \in \partial\Omega,\; 0 < r < \operatorname{diam}(\partial\Omega).
\end{equation}
 
Due to the fact that Euclidean cubes can be made to form a disjoint union of a given larger Euclidean cube, with appropriately scaled radii, they are more natural to work with instead of balls in our argument, as will become apparent in the course of the proof.

Our arguments are more direct and elementary, and only uses the relationship between the harmonic measure and the Dirichlet Green function in a given domain as in \cref{imp2}, and elementary potential theoretic considerations.

\section{Notation}

Throughout the proof, $C_i>0$ for any $i\geq 1$, denotes a constant whose value may change from line to line. 
When necessary, we indicate the dependence of $C$ on ambient parameters.

For any $Q\in \Delta_j$, we denote by $l(Q)=2^j$, with the value of $j$ from \cref{cubelemma}, and we denote by $r(B)$ the radius of a ball $B(x,r)$ of radius $r$ centered at $x$. We denote by $\frac{1}{C}B$ the ball withthe same radius as $B$ but with $1/C$ times the radius of $B$. By $S(x,r)$ we denote the sphere of radius $r$ centered at $x$.

We borrow the notation from \cite{Az20}, for a Radon measure $\mu$, $s\geq 0$, and a ball $B$, we define
\begin{align}
\Theta_{\mu}^{s}(B)= \frac{\mu(B)}{r(B)^{s}}.
\end{align}

We also use the following notation to define the density of the harmonic measure for any one of the cubes of \cref{cubelemma}, as,
\begin{align}
    k^{t}(Q):=\frac{\omega^{t}_{\Omega}(Q)}{\sigma(Q)}.
\end{align}

When necessarily, we also use a subscript, for the harmonic measure, to denote the domain in which we are working, and we also use a superscript when necessary to denote the pole of the harmonic measure. This should be understood from context. When no subscript is used, it is understood that the domain in consideration is $\Omega$.

We also distinguish between Euclidean cubes with faces parallel to the coordinate planes, such as $\{(x_1,\dots,x_n): b_i\leq x_i\leq a_i, \text{for all}\ 1\leq i\leq n.  \}$, and the dyadic cubes of David and Semmes\cite{DS1,DS2} mentioned in Theorem 3.1. As mentioned earlier, $C(x,r)$ is the Euclidean cube given by $C(x,r):=\{(y_1,\dots,y_n):x_i -r\leq y_i\leq x_i +r , 1\leq i\leq n\}$, where $x=\{x_1,\dots,x_n\}$. Depending on the context, we will also consider $(n-1)$-dimensional cubes contained it hyperplanes. 

When $2^{j}\leq C_2<2^{j+1}$ for some $j\in \Z$, we write $C_1\approx C_2$ with $C_1=2^j$. Throughout this paper, by a covering of any set $E$, we mean a finite set of elements indexed by $J$,  $\{E_j\}|_{j\geq J}$ so that we have the disjoint union $E=\sqcup_{j\geq 1}E_j$ (i.e. for any two $E_j, E_k$ with $j\neq k$, $j,k\in J$ , we have $E_j \cap E_k=\phi$). Finally, we denote the null set by $\phi$.

\subsection{Table of notation}
\label{sec:notation-table}
For ease of reference, we collect the notation used throughout the paper below.

%

\begin{longtable}{@{}p{0.20\textwidth}p{0.72\textwidth}@{}}
\toprule
\textbf{Symbol} & \textbf{Meaning} \\
\midrule
\endfirsthead
\multicolumn{2}{@{}l}{\textit{(continued)}} \\
\toprule
\textbf{Symbol} & \textbf{Meaning} \\
\midrule
\endhead
\bottomrule
\endfoot

\multicolumn{2}{@{}l}{\textbf{Domain and boundary}} \\[2pt]
$\Omega \subset \mathbb{R}^n$ & the domain under consideration, $n \geq 3$ \\
$\partial \Omega$ & boundary of $\Omega$ \\
$s$ & Ahlfors--David dimension of $\partial\Omega$, $n-1-\delta_0 \leq s \leq n-1$ \\
$\delta$ & $\delta := n-1-s$, codimension defect from $n-1$ \\
$\delta_0$ & explicit threshold on $\delta$ below which the dimension-drop argument applies \\
$\beta$ & uniform non-flatness parameter (Eq.~(1)) \\
$\beta_1$ & $\beta_1 := 1/(4N_\beta+2)$ \\
$N_\beta$ & $N_\beta := \lceil 1/\beta \rceil$ \\
$V$ & an $(n-1)$-dimensional plane in $\mathbb{R}^n$ \\
$b^\beta_\omega(x,r)$ & uniform non-flatness coefficient at $x$, scale $r$ (Eq.~(1)/(1$'$)) \\
$K$ & exceptional set of dimension $< s$ carrying full harmonic measure \\

\addlinespace
\multicolumn{2}{@{}l}{\textbf{Measures and densities}} \\[2pt]
$\sigma$ & surface measure, $\sigma := \mathcal{H}^s|_{\partial\Omega}$ \\
$\omega = \omega_\Omega$ & harmonic measure of $\Omega$ (pole suppressed when clear from context) \\
$\omega^x,\ \omega^p$ & harmonic measure with pole $x$ (resp.\ $p$) \\
$\widetilde\omega,\ \omega_{\Omega\setminus B_0}$ & harmonic measure of the subdomain $\Omega\setminus B_0$ \\
$\dim(\mu)$ & Hausdorff dimension of a Borel measure $\mu$ \\
$\Theta^s_\mu(B)$ & $:= \mu(B)/r(B)^s$, normalized density of $\mu$ on ball $B$ \\
$\Theta(Q),\ \Theta_{\omega^x}(Q)$ & shorthand for $\Theta^s_{\omega^x}(Q)$, density of $\omega^x$ on cube $Q$ \\
$k^t(Q)$ & $:= \omega^t_\Omega(Q)/\sigma(Q)$, density of the harmonic measure with pole $t$ on $Q$ \\

\addlinespace
\multicolumn{2}{@{}l}{\textbf{Points and distances}} \\[2pt]
$p,\ p_0,\ p_1$ & poles of harmonic measure / interior points of $\Omega$ \\
$q_1 = \widehat{p_1}$ & point of $\partial\Omega$ nearest to $p_1$ \\
$r_1$ & $:= |p_1 - q_1|$ \\
$\widehat{x}$ & point of $\partial\Omega$ nearest to $x$ \\
$\delta(x),\ \delta_\Omega(x)$ & distance from $x \in \Omega$ to $\partial\Omega$ \\
$A(x,r)$ & interior corkscrew point relative to the surface ball $\Delta(x,r)$ \\
$c(Q)$ & ``center'' of the dyadic cube $Q$ (Theorem~3.2) \\

\addlinespace
\multicolumn{2}{@{}l}{\textbf{Cubes, balls, and dyadic structure}} \\[2pt]
$B(x,r)$ & Euclidean ball of radius $r$ centered at $x$ \\
$C(x,r)$ & Euclidean cube $\{y : |y_i - x_i| \leq r,\ 1\le i \le n\}$ \\
$\Delta(x,r)$ & surface ball, $\Delta(x,r) := B(x,r)\cap\partial\Omega$ \\
$\Delta_j$ & the David--Semmes dyadic decomposition of $\partial\Omega$ at scale $j$ \\
$l(Q)$ & sidelength of the dyadic cube $Q$, $l(Q) = 2^j$ \\
$B_Q$ & surface ball associated to $Q$, $B_Q := \Delta(c(Q), C_1^{-1}\operatorname{diam}Q)$ \\
$Q,\ U$ & dyadic cubes; $U$ typically denotes a cube containing $c(Q)$ with $l(U) = C'^{-1}l(Q)$ \\
$\kappa$ & dyadic scaling ratio, $0 < \kappa \leq 1/2$, fixed via Eq.~(17) \\
$A_m(q_1)$ & $m$-th annular region, $:= C(q_1,\kappa^m t_1)\setminus C(q_1,\kappa^{m+1}t_1)$ \\
$t_1$ & radius of the initial cube $C(q_1,t_1)$, small relative to $r_1$ (Eq.~(14)) \\

\addlinespace
\multicolumn{2}{@{}l}{\textbf{Planes and normal vectors}} \\[2pt]
$\Sigma_1,\ \Sigma_2,\ \Sigma_3$ & auxiliary hyperplanes through $q_1$ (or $q_2$) used in the gradient argument \\
$H_1^+,\ H_1^-$ & half-spaces determined by $\Sigma_1$ \\
$\widehat{v}_1,\ \widehat{u}_2,\ \widehat{\nu}$ & unit normal vectors to $\Sigma_1,\ \Sigma_2,\ \Sigma_3$ respectively \\

\addlinespace
\multicolumn{2}{@{}l}{\textbf{Green's function and potential theory}} \\[2pt]
$G(X,Y)$ & Green's function for $\Omega$ (Eq.~(3)) \\
$E(X)$ & $:= c_n|X|^{2-n}$, fundamental solution of the Laplacian in $\mathbb{R}^n$ \\
$E(x,y)$ & contribution at $x$ from the fundamental solution with pole at $y$ \\
$\delta_1 G,\ \delta_2 G$ & boundary-measure and pole contributions to the potential difference (Eq.~(21),(22)) \\
$I_{near},\ I_{mid},\ I_{far}$ & near-, intermediate-, and far-field contributions to $\delta_1 G$ (Eq.~(21)) \\
$\nabla_{A_m}$ & gradient at $q_1$ due to the harmonic measure in $A_m$ \\

\addlinespace
\multicolumn{2}{@{}l}{\textbf{Iteration parameters and auxiliary constants}} \\[2pt]
$M,\ M_1,\ M_2$ & density growth factors; $M$ large and given, $M_1,M_2$ derived from it \\
$w$ & integer exponent controlling the separation from $\Sigma_1$ (Lemma~5, depends on $\beta$) \\
$\gamma$ & sub-cube scaling parameter, $0 < \gamma \leq 1/2$ \\
$N,\ N_1,\ N_2$ & iteration counts / length-scale exponents fixed in the course of the proof \\
$\theta_0,\ \theta_1$ & separation parameters with $\theta_0 = \beta_1^{wn} \leq \theta_1 \leq 1$ \\
$\eta,\ \eta_1,\ \eta_2$ & proportions of surface measure occupied by a sub-cube (Lemma~7) \\
$j_0$ & length-scale exponent of the dyadic cube produced by the main density-increment argument \\
$\chi$ & constant appearing in the conclusion of Theorem~5.1 \\
$\epsilon$ & small parameter bounding $\delta_\Omega(x)$ from below in Theorem~5.1 \\
$\alpha$ & H\"older exponent in the CDC estimate (Lemma~3) \\
$\tau$ & small-boundary parameter in the dyadic grid property, Theorem~3.1(4) \\

\end{longtable}

\section{Proof of the theorem.}

Without loss of generality, in this paper we consider the pole $p$ of the harmonic measure to be located at some corkscrew point $A(x,r_0)$. All the balls and dyadic cubes under subsequent consideration have diameters less than this fixed $r_0$.




    For the following argument, for any point $p_1 \in \Omega$, we consider a point $q_1=\widehat{p_1}$ on the boundary $\partial\Omega$ that is closest to $p_1$, and $|p_1 -q_1|=r_1$

      We consider this 'corner' point $q_1$, and a ball of radius $t_1$ sufficiently small in comparison to $r_1$ so that for a fixed large positive integer $w$, we have, 
      \begin{align}\label{separate}
      \frac{t_1}{r_1}= c (\beta_1)^{w(1+s)}< c(\beta_1)^{w}.
      \end{align}
      
      for some sufficiently small uniform constant $c\ll 1$, and a large positive constant $w$ also dependent on $\beta$, to be determined later.

      \begin{lemma}
          Consider the cube $C(q_1,t_1)$. Then there exists a $0 < \kappa\leq \frac{1}{2}$ so that we have the following decomposition of the ball $C(q_1,t_1):=\cup_{m\geq 0}C(q_1,\kappa^m t_1)\setminus C(q_1,\kappa^{m+1}t_1) :=A(q_1,m)$, so that, 
        \begin{align}\label{eq10}
    \frac{1}{2}  C_{1}^{-1} \kappa^{ms}t_{1}^{s}\leq \sigma(A_m (q_1))\leq C_1 \kappa^{ms} t_{1}^s
\end{align}
      \end{lemma}

      \begin{proof}Consider the following decomposition of this cube $B(q_1,t_1)$: depending on the Ahlfors David regularity constant $C_1$. We consider the value $0<\kappa<1$ so that,
      \begin{align}\label{ample}
            C_{1}^{-1} t_{1}^s -C_{1}(\kappa t)^s \geq \frac{1}{2}  C_{1}^{-1} t_{1}^{s}\Leftrightarrow \kappa^s  \leq \frac{1}{2}C_{1}^{-2}.
      \end{align}

     Consider the unique positive integer $N$ so that, $\frac{1}{N+1}<C_{1}^{-2/s}\leq \frac{1}{N}$. Thus, $N=\lfloor C_{1}^{2/s} \rfloor$. We choose a value of 
     \begin{align}\label{kappa}
    \kappa=\frac{1}{42n}= \frac{1}{4\lfloor C_{1}^{2/s} \rfloor +2}.
    \end{align}
     In particular, we have $\kappa\leq \frac{1}{2}$.
     
     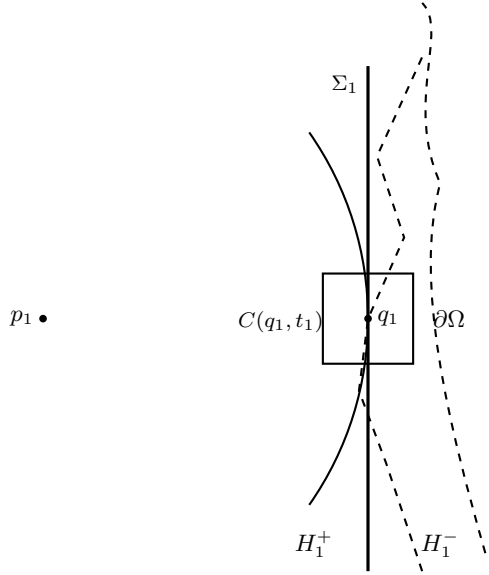
\begin{figure}
     \centering

\begin{tikzpicture}[
    scale=1.2,
    every node/.style={font=\small}
]

\def\r{3.6}
\def\tone{1}

\coordinate (p1) at (0,0);
\coordinate (q1) at (\r,0);

\draw[thick]
  ({\r*cos(-35)}, {\r*sin(-35)})
  arc[start angle=-35, end angle=35, radius=\r];

\draw[very thick] (\r,-2.8) -- (\r,2.8);
\node[above left, font=\footnotesize] at (\r,2.4) {$\Sigma_1$};


\draw[dashed, thick]
  (\r+0.6,-2.8)
    -- (\r+0.2,-1.6)
    -- (\r-0.1,-0.8)
    -- (q1)
    -- (\r+0.4,0.9)
    -- (\r+0.1,1.8)
    -- (\r+0.6,2.9);

\draw[dashed, thick]
  (\r+1.3,-2.6)
    .. controls (\r+0.9,-1.2) and (\r+0.5,0.4) ..
  (\r+0.8,1.5)
    .. controls (\r+0.4,2.4) and (\r+0.9,3.2) ..
  (\r+0.6,3.5);

\fill (p1) circle (1.2pt);
\node[left] at (p1) {$p_1$};

\fill (q1) circle (1.2pt);
\node[right] at (q1) {$q_1$};

\node at (\r-0.6,-2.5) {$H_1^{+}$};
\node at (\r+0.8,-2.5) {$H_1^{-}$};
\node at (\r+0.9,0) {$\partial\Omega$};

\draw[thick]
  (\r-0.5,-0.5) rectangle (\r+0.5,0.5);

\node[above right, font=\footnotesize]
  at (\r-1.54,-0.25) {$C(q_1,t_1)$};

\end{tikzpicture}

\caption{The corkscrew point $p_1$ and a point $q_1\in\partial\Omega$ nearest to $p_1$ is shown in the figure. The hyperplane $\Sigma_1$ passes through $q_1$ and is perpendicular to the line joining $p_1,q_1$, and the cube $C(q_1,t_1)$ is also shown. The left and right half planes in the figure are respectively $H^{+}_1, H^{-}_1$. The arc shows a portion of the sphere $S(p_1, r_1)$. The dashed lines depict the boundary $\partial\Omega$, which lies in the complement of the ball $B(p_1,r_1)$.}
\end{figure}

This ensures that for the annular region $\sigma(C(q_1,t_1)\setminus C(q_1,\kappa t_1))\geq \frac{1}{2}  C_{1}^{-1} t_{1}^{s} $. Further, we trivially have an upper bound of $\sigma(C(q_1,t_1)\setminus C(q_1,\kappa t_1))\leq C_1 t_{1}^s$, for the amount of surface measure within this region. Thus we have, 
\begin{align}\label{eq6}
    \frac{1}{2}  C_{1}^{-1} t_{1}^{s}\leq \sigma(C(q_1,t_1)\setminus C(q_1,\kappa t_1))\leq C_1 t_{1}^s
\end{align}

Successively, we consider for each integer $m\geq 0$, the annular regions, 
\begin{align}
    A_m (q_1):= C(q_1,\kappa ^m t_1)\setminus C(q_1,\kappa^{m+1} t_1).
\end{align}

By the same argument as that for \cref{eq6}, we also get that,
\begin{align}\label{eq10}
    \frac{1}{2}  C_{1}^{-1} \kappa^{ms}t_{1}^{s}\leq \sigma(A_m (q_1))\leq C_1 \kappa^{ms} t_{1}^s
\end{align}\end{proof}

Within each such annulus $A_m$, consider the gradient at the point $q_1$ due to the surface measure within $A_m$. We use the condition arising from \cref{ample} above, the \cref{bilateral} conditions.

Consider the difference in the potential between the points $q_1, q_2\in \partial\Omega$ due to the harmonic measure on $\partial\Omega$.

\begin{multline}\label{decomposition}
\delta_1 G(q_1,q_2)
=
\underbrace{\int_{\mathcal{C}(B(q_1,\kappa^N t_1))} \bigl(\mathcal{E}(q_1,y)-\mathcal{E}(q_2,y)\bigr)\,d\omega(y)}_{I_{\mathrm{near}}}
+
\underbrace{\int_{\mathcal{C}(B(q_1,t_1)\setminus \mathcal{C}(B(q_1,\kappa^{N}t_1)} \bigl(\mathcal{E}(q_1,y)-\mathcal{E}(q_2,y)\bigr)\,d\omega(y)}_{I_{\mathrm{mid}}}
\\+
\underbrace{\int_{\partial\Omega\setminus \mathcal{C}(B(q_1,t_1)} \bigl(\mathcal{E}(q_1,y)-\mathcal{E}(q_2,y)\bigr)\,d\omega(y)}_{I_{\mathrm{far}}}.
\end{multline}

 Recall that $\mathcal{E}(x,y)$ is the contribution at $x$, due to the fundamental solution with the pole at $y$. Here we have considered the sets $ \mathcal{C}(B(q_1,t_1)), \mathcal{C}(B(q_1,\kappa^N t_1))$ as the minimal covers of $B(q_1, t_1),B(q_1,\kappa^N t_1)$ respectively, by dyadic cubes of length $2^{j_0}$ for the $j_0$ to be determined.

 Along with this, we have a contribution $\delta_2 G(q_1,q_2)=-\mathcal{E}(q_1,p_0)+\mathcal{E}(q_2,p_0)$, due to the pole itself located at $p_0$, to the difference in the potential between the points $q_0,q_1$. Since we have Dirichlet boundary conditions, we then must have 
 \begin{align}
     \delta_1 G(q_1,q_2)+\delta_2 G(q_1,q_2)=0
 \end{align}
By construction $p_0$ will lie outside all these balls $B(q_1,t_1)$ under consideration, and we will treat the contribution due to the pole at $p_0$ alongside the contribution from $I_{far}$, as will be clear from the argument.
 
 We estimate separately, the near-field term $I_{\mathrm{near}}$,
the intermediate contribution $I_{\mathrm{mid}}$,
and the far-field term $I_{\mathrm{far}}$ along with $\delta_2G(q_1,q_2)$. 
 
First consider the contribution from $I_{mid}$. For each of the dyadic cubes contained in each of the annular regions in $I_{mid}$, the contribution due to the gradient term will dominate the second order contribution from these dyadic cubes.
 
 Note a trivial upper bound to the contribution to the gradient at $q_1$ due to the harmonic measure in $A_m$: since the total harmonic measure in the annular region $A_m$ is trivially upper bounded by $MC_{1} (\kappa^m t_{1})^s$, we get a contribution to the gradient whose magnitude is upper bounded by,
 \begin{align}\label{upperbound}
     \frac{MC_{2} \kappa^{ms} t_{1}^s}{(\kappa^{m} t)^{n-1}}.
 \end{align}
    for some constant $C_2$. 

   Consider the $(n-1)$- dimensional plane perpendicular to the line joining the points $p_1,q_1$, and passing through $q_1$ . Call this plane $\Sigma_1$. We write $H_{1}^{+}$ for the half plane with the unit outward normal vector $\hat{v_1}$ pointing from $q_1$ towards $p_1$, and the other half plane as $H_{1}^{-}$. 
   
    Now we are interested in a lower bound to the magnitude of the gradient due to $A_m$, at $q_1$. We we use the condition \cref{bilateral} for this.

    First note that if the entire surface measure was concentrated arbitrarily close to $A_m \cap \Sigma_1$, then one might have enough cancellations in the contribution to the gradient at $q_1$, along the plane $\Sigma_1$, and thus get a negligible contribution to the gradient. 
    
    However, using the uniform non-flatness condition of \cref{bilateral,adr}, along with \cref{separate}, we prove the following. 
    
\begin{lemma}\label{lemma6}
    In the case where $s=n-1$, there exists a fixed positive integer $w$ dependent on $\beta$, so that for each $m\geq 0$ there exists a point in $\partial\Omega\cap H^{-}_{1}\cap A_m$ which is at a distance at least $\beta_1^w \kappa^{m+1} t_1$ from $\Sigma_1$. 
    
    
    Here, we have,
    \begin{align}
           w=\max \{1, \Big\lfloor \frac{-\log 2 -2\log C_2 -\log((\frac{1}{\kappa})^{n-1}-1)+(n-1)\log \frac{1}{\kappa} }{\log(1-\beta_1^{n-1})}\Big\rfloor \}.
       \end{align}

\end{lemma}
    
In the case where $s<n-1$, we will later modify this proof, in particular modify \cref{eq20}.
\begin{proof}[Proof of \cref{lemma6}]
For each $m\geq 1$, consider the disjoint union $\mathcal{D}_m$ of $A_m\cap \Sigma_1 $ by Euclidean cubes of radius $ \kappa^{m+1} t_1$. Note from \cref{kappa} that we have chosen $\kappa=\frac{1}{4\lceil C_{1}^{2/s}\rceil+2}$, and such a covering always exists. More precisely, we have a disjoint covering $\mathcal{F}_m$ of $C(q_1,\kappa^m t_1)$ by cubes of radius $ \kappa^{m+1} t_1$ which includes the cube $C(q_1,\kappa^{m+1}t_1)$, and we have $\mathcal{D}_m=\mathcal{F}_m\setminus \{C(q_1,\kappa^{m+1}t_1)\}$.

If there exists a point in $\partial\Omega\cap H^{-}_{1}\cap A_m$ which is a distance at least $\frac{\beta}{2} \kappa^{m+1} t_1$ away from $\Sigma_1$, then we are done.
    
    Otherwise, all points in $\partial\Omega\cap H^{-}_{1}\cap A_m$ lie at most a distance $\frac{\beta}{2} \kappa^{m+1} t_1$ from $\Sigma_1$. Then, we claim that for any $C(q_2,\kappa^{m+1} t_1)\in \mathcal{D}_m$, with some center $q_2$, we have
    \begin{align}\label{eq11}
        \sup\limits_{y\in \Sigma_1 \cap C(q_2,\kappa^{m+1} t_1 ) }\frac{\text{dist}(y,\partial\Omega)}{\kappa^{m+1} t_1} \geq \frac{\beta}{2}.
    \end{align}

    To show this, note first that the ball $B(p_1,r_1)$ is empty, and the condition \cref{separate}, forces the part of $\partial\Omega$ in $H^{+}_{1}\cap A_m$ to be quantifiably closer to $\Sigma_1$ than $\frac{\beta}{2}\kappa^{m} t_1$. More precisely, an elementary calculation shows that the maximum distance of $\partial\Omega\cap H^{+}_1 \cap C(q_1,\kappa^m t_1)\cap B
    (p_1,r_1)^c$ to the plane $\Sigma_1$, where $B(p_1,r_1)^c$  is the complement of the ball $B(p_1,r_1)$, is given by $\kappa^m t_{1}^{2}/r_1$ and we have because of \cref{separate}, 
    \begin{align}\label{eq20'}
    \frac{\kappa^m t_{1}^2}{r_1}\ll \frac{\beta^w}{2} \kappa^m t_1  \ll \frac{\beta}{2} \kappa^m t_1.
    \end{align}

Thus, considering the condition \cref{bilateral} with the plane $\Sigma_1$, and the cube $C(q_2,\kappa^{m+1} t_1)\in \mathcal{D}_m$, the hypothesis that all points in $\partial\Omega\cap H^{-}_{1}\cap A_m$ lie at most a distance $\frac{\beta}{2} \kappa^{m+1} t_1$ from $\Sigma_1$, and \cref{eq20'} and the discussion preceding it, we have for the corresponding first term on the right of the expression in \cref{bilateral}, 
\begin{align}
    \sup_{y \in C(q_2,\kappa^{m+1} t_1)\cap \partial \Omega} \frac{\operatorname{dist}(y,\Sigma_1)}{\kappa^{m+1} t_1}\leq  \frac{\beta}{2},
\end{align}

and thus for the second term on the right of \cref{bilateral} we have the condition of \cref{eq11}. 


Thus, from \cref{eq11}, we get that there exists some point $y\in  C(q_2,\kappa^{m+1} t_1)\subset \Sigma_1$ so that the cube $C(y,\kappa^{m+1} t_1 \beta/2)$ does not contain any point of $\partial\Omega$\footnote{Here, $c<1$ is a geometric factor dependent on the dimension $n$.}. Here, we mean $(n-1)$ dimensional cubes contained in $\Sigma_1$. Henceforth, this should be clear from context. Note that here and in subsequent iterations of this argument in this lemma, it might happen that this chosen point $y$ is on the boundary of a cube from the previous generation that is already empty. Even if $y$ is in the interior of the cube $C(q_2,\kappa^{m+1}t_1)$, we want to only discard a subcube within $C(y,\kappa^{m+1} t_1 \beta/2)\cap C(q_2,\kappa^{m+1}t_1)$, of comparable size to $C(y,\kappa^{m+1} t_1 \beta/2)$, that belongs to a covering of $C(q_2,\kappa^{m+1} t_1)$.

We have some cube $C(q'', \kappa^{m+1} t_1 \en)\subset C(y,\kappa^{m+1} t_1 \beta/2)\cap C(q_2,\kappa^{m+1}t_1)$ with $q''\in \Sigma_1$, that does not contain any point of $\partial\Omega$. Further, because of the choice of $\en$, $ C(q'', \kappa^{m+1} t_1 \en)$ is an element of a covering of $C(q_2,\kappa^{m+1}t_1)$, with radius $\kappa^{m+1}t_1 \en$. Given the cube $C(q_2,\kappa^{m+1} t_1)$, we denote the set of such cubes, of radius $\kappa^{m+1}t_1 \en$, in this covering of $C(q_2,\kappa^{m+1} t_1)$, as $\mathcal{D}_{C(q_2,\kappa^{m+1} t_1)}$.

Upon a disjoint union over all these cubes in $\mathcal{D}_m$ we get a covering of $A_m \cap \Sigma_1$ by cubes of radius $\kappa^{m+1} t_1 \en$,
 \begin{align}
     \mathcal{D}_{m}^{(1)}:=\cup_{C(q_2,\kappa^{m+1}t_1)\in \mathcal{D}_m}\mathcal{D}_{C(q_2,\kappa^{m+1}t_1)}.
 \end{align}

  We discard all such cubes of the form $ C(q'', \kappa^{m+1} t_1 \en)$ chosen in the previous paragraph, by repeating the same argument over all the cubes of the form $C(q_2,\kappa^{m+1}t_1)\in \mathcal{D}_m$. We consider the disjoint union of boundary cubes that remains after this process, within $\Sigma_1\cap A_m$, and call that $\Sigma_2$. In particular, the cube $C(q'', \kappa^{m+1} t_1 \en)$ belongs to $\mathcal{D}_{m}^{(1)}$, and $C(q'', \kappa^{m+1} t_1 \en)\cap \Sigma_2 =\phi$ . 
 

Now, if there is a point in $\partial\Omega\cap H^{-}_1 \cap A_m$ at a distance at least $\beta \en \kappa^{m+1} t_1/2 \geq (\en)^2 \kappa^{m+1} t_1 $ away from $\Sigma_1$, then also we are done. Otherwise, assume that all points of $\partial\Omega\cap H^{-}_1 \cap A_m$ are at a distance at most $ \kappa^{m+1} t_1\en \beta/2$ from $\Sigma_1$.

In other words, we assume that, for any cube $C(y_i, \kappa^{m+1} t_1\beta_1)\in \mathcal{D}^{(1)}_m$, we have
\begin{align}
    \sup\limits_{y\in C(y_i, \kappa^{m+1} t_1\beta_1)\cap \partial\Omega} \frac{\text{dist}(y,\Sigma_1)}{\kappa^{m+1} t_1\beta_1}\leq \frac{\beta}{2}.
\end{align}


Thus, any point of $\partial\Omega\cap H^{-}_1 \cap A_m$ now belong to the closure of some cube $C(y_i, \kappa^{m+1} t_1\beta_1)$ belonging to this collection $\mathcal{D}_{m}^{(1)}$. Then for any given cube $C(y_i, \kappa^{m+1} t_1 \beta_1)\in \mathcal{D}_{m}^{(1)}$  we repeat the previous argument and discard a set of cubes of radius $ \kappa^{m+1} t_1 \beta^{2}_1$. In the process we find a disjoint collection of cubes $\mathcal{D}_{m}^{(2)}$, each now of radius $c^2 \kappa^{m+1} t_1\beta^{2}_1$, each centered on $\Sigma_1\cap A_m$, and which thus covers the remaining union of cubes that have not been discarded in $\Sigma_1\cap A_m$, and which we define as $\Sigma_3$.

Thus, by \cref{bilateral}, we get that there exists some point $y\in \Sigma_3$ so that 
\begin{align}
    \text{dist}(y,\partial\Omega)\geq \frac{\kappa^{m+1} t_1 \beta\en}{2} .
\end{align}
This gives us, as before, that a cube of the form $C(q_3, \kappa^{m+1} t_1 \beta_{1}^{2})$  does not contain any point of $\partial\Omega$, and further that this cube $C(q_3, \kappa^{m+1} t_1 \beta_{1}^{2})$ belongs to a covering of $\Sigma_3$.


Now we repeat this argument inductively. Note in particular that the argument also ensures that the intersection of these discarded cubes with $H^{+}_1$ is also empty, because the points of $\partial\Omega\cap H^{+}_1$ are quantifiably closer to $\Sigma$ compared to the radius of the cubes under consideration till the stage $w$, due to \cref{eq20',separate}.

We continue this argument until we reach a stage where we violate the Ahlfors regularity property of \cref{adr}. Note that we have after the $w$'th stage of the argument, by hypothesis, that all the points of $A_m\cap H^{-1}_1\cap \partial\Omega$ lie within a distance at most $\beta \beta_1^{w-1}\kappa^{m+1} t_1/2$ from $\Sigma_1$. If there is a point at least a distance $ \beta_1^{w}\kappa^{m+1} t_1/2$ away from $\Sigma_1$ within $H^{-1}_{1}$ then we are done, otherwise assume that all these points lie within a distance $ \beta_1^{w}\kappa^{m+1} t_1/2$ from $\Sigma_1$. Note that the points in $H^{+}_{1}$ are in any case closer to $\Sigma_1$ because of \cref{eq20',separate}.

Thus we choose $w$ big enough so that we have,
\begin{multline}\label{eq20}
    CC_1\Big(\frac{(\kappa^{m}t_1)^{n-1}}{(\kappa^{m+1}t_1)^{n-1}}-1\Big) \Big(\frac{1}{\beta_1^{n-1}}-1\Big)^w\Big( \beta_1^w \kappa^{m+1} t_1 \Big)^{n-1}=  CC_1\Big(\Big(\frac{1}{\kappa}\Big)^{n-1}-1\Big) \Big(\frac{1}{\beta_1^{n-1}}-1\Big)^w\Big( \beta_1^w \kappa^{m+1} t_1 \Big)^{n-1}\\ \leq \frac{1}{2}C_{2}^{{-1}} (\kappa^{m} t_1)^{n-1}.
\end{multline}

This gives us a failure of the $s-$Ahlfors-David regularity for the cube $C(q,\kappa^{m} t_1)$, and thus a contradiction. The first factor on the left is gives us the number of cubes in $C(q_1,\kappa^{m}t_1)\setminus C(q_1,\kappa^{m+1}t_1)$ excluding $C(q_1,\kappa^{m+1}t_1)$, in the disjoint covering of $C(q_1,\kappa^{m}t_1)\setminus C(q_1,\kappa^{m+1}t_1)$  by cubes of radius $\kappa^{m+1}t_1)$. The second factor on the left gives us the number of subcubes of radius $(\beta_{1}^{w} \kappa^{m+1}t_1)$ within any given cube $C(q'', \kappa^{m+1}t_1)\cap \Sigma_1$ of the disjoint cover of $C(q_1,\kappa^{m}t_1)\setminus C(q_1,\kappa^{m+1}t_1)$, after the $w$'th stage of the iteration. The remaining factor on the left hand side of \cref{eq20} gives us an upper estimate on the measure of $\partial\Omega\cap C(q''',\beta_{1}^{w}\kappa^{m+1}t_1)$, up to a constant factor\footnote{Note that in general, the upper estimate in \cref{adr} can be taken around balls with centers not necessarily lying on a point on the boundary $\partial\Omega$, up to an additional uniform constant. In this instance, we have taken the estimate on cubes centered on points on $\Sigma_1$.}.

Here, the left hand quantity is an upper bound for the total amount of surface measure contained in $H^{-}_1 \cap A_m$, by the above iteration done $w$ many times. The right hand side is a lower bound for the amount of surface measure in $A_m$, that we found from \cref{eq10}.
       Thus for $w$ sufficiently large, noting that $C_{1}\geq 1$, we ensure that
       \begin{multline}
          \kappa^s \Big(\big(\frac{1}{\kappa}\big)^{n-1}-1\Big) (1-\beta_1 ^{n-1})^w \cdot \beta_1 ^{w(s-n+1)}\leq \frac{1}{2}C_{2}^{-1} \\ \Leftrightarrow w\log(1-\beta_1^{n-1})  \leq \log (\frac{1}{2}C_{2}^{-1}) -(n-1)\log \kappa -\log\Big((\frac{1}{\kappa})^{n-1}-1\Big)
       \end{multline}

       where we have incorporated the previous constants into the constant $C_2$.
       
       Thus we have, 
       \begin{align}
       (n-1)\log \kappa+ \log\Big(\big(\frac{1}{\kappa}\big)^{n-1}-1\Big)+   w(\log(1-(\en)^{n-1})\leq  \log (\frac{1}{2}C_{2}^{-2}) 
       \end{align}
       
       Thus, it is enough to choose, 
       \begin{align}\label{w estimate}
           w= \lceil \frac{-\log 2 -2\log C_2 -\log((\frac{1}{\kappa})^{n-1}-1) +(n-1)\log \tfrac1\kappa}{\log(1-\beta_1^{n-1})}\rceil .
       \end{align}

      Recall that $\kappa$ is taken as a function of the Ahlfors-David regular constant through \cref{kappa}. Thus we have explicit bounds on $w$ in terms of the ambient parameters, and this concludes the proof of the lemma.
    \end{proof}

Note that the second term in the denominator in \cref{w estimate} is precisely $0$ in the case of $s=n-1$, which we deal with first. In that case, we have an easy estimate for $w$ in terms of $\en$ and in turn $\beta$.  Later, for the case of $n-1-\delta_0 <s<n-1$, we use \cref{w estimate} in part, to estimate the optimal $\delta$, in a later part of the argument.
 
      Thus, we must have  a point $x_m$ of $\partial\Omega\cap A_m$ that is a distance at least $\beta_1^w \kappa^m t_1$ from $\Sigma_1$. Thus, by Ahlfors regularity, with this condition, we will have a surface measure at least $C_{1}^{-1}(\frac{\beta_1^w\kappa^m t_1}{2})^{n-1}$ in the ball $B(x_m, \frac{\beta_1^w)\kappa^m t_1}{2})$.

      In this case, the magnitude of the gradient due to this ball $B(x_m, \frac{\beta_1^w)\kappa^m t_1}{2})$, at $q$, is bounded from below by, 
      \begin{align}\label{lowerbound}
          C_{3}M^{-1}\Big(\frac{\beta_1^w\kappa^m t_1}{2}\Big)^{n-1} \Big(\frac{1}{\kappa^m t_1}\Big)^{n-1}.
      \end{align}
      for some constant $C_3$.

     \begin{lemma}\label{lemma9}
Fix $\eps$, an arbitrarily small constant. Given $n\geq 3$, and $C$ large enough dependent on the parameters in \cref{cubelemma}(3), there exists $M> 1$ dependent on the parameters $C_1$ and $c_0$ (see \cref{r:bourgain}), there exists a constant $\chi>0$ , so that the following holds.
Let $n-1-\delta_0 \leq s \leq n-1$ and let $\Omega \subset \mathbb{R}^{n}$ be a connected domain with $s$-Ahlfors-David regular boundary, such that
\cref{adr,bilateral} hold.
Let $Q\in \partial\Omega$ be a dyadic cube, with center $c(Q)$, obtained from the decomposition given by \cref{cubelemma}. Let  $U$ be a dyadic cube containing $c(Q)$, with $l(U)=\frac{1}{C'}l(Q)$, with $C'\approx C$, $C' >1$.  Then given any $x\in \frac{1}{3} B_U\setminus \partial\Omega$, with $\delta_\Omega(x)\geq \eps l(Q)$, we have some dyadic cube $Q_1\subset Q$ so that, $l(Q_1)\geq \chi l(Q)$, and
\begin{align}\label{eq366}
   \Theta_{\omega^x}(Q_1)=\frac{\omega^{x}(Q_1)}{\sigma(Q_1)}\geq M\frac{\omega^x(Q)}{\sigma(Q)}=M\Theta_{\omega^x}(Q).
\end{align}

Here, throughout we write $\Theta=\Theta_{\omega^{x}}^{s}$.
\end{lemma}
      
   The value of $M$ is quantified in \cref{eq122} .
   \begin{proof}  First note that $\Theta(U)$ and $\Theta(Q)$ are comparable, depending on the constant $C$. This is because, since the pole $x\in \frac{1}{3}B_U\setminus\Omega$, we have,
    \begin{align}\label{eq37}
      \frac{1}{\sigma(U)} \geq \Theta(U)\geq \frac{c_0}{\sigma(U)}, \qquad \frac{1}{C_1\sigma(U)} \geq \Theta(Q)\geq \frac{c_0}{C_1 \sigma(U)},
    \end{align}

    for some constant $C_1$ dependent on the constant $C$. Throughout the proof, we work with $\Theta(U)$, and culminate the argument by replacing $\Theta(U)$ with $\Theta(Q)$, using \cref{eq37}, up to an additional constant factor, for all sufficiently large $M$.

      Consider $\widehat{x}$ which is a point on $\partial\Omega$ nearest to $x$. In this case, the Bourgain estimate gives us, 
            \begin{align}\label{eq200}
          \omega^x(U )=c_{U}\geq c_0,
      \end{align}

for some uniform $c_0$. Note that depending on the ambient parameters, choosing $C$ large enough, using \cref{center}, such a dyadic cube $U$ can always be found.

Consider the corkscrew point $A(U)$ relative to the ball $\frac{1}{3}B_U$, and a point $\widehat{A(U)}\in \frac{2}{3}B_U$ nearest to $A(U)$, and we denote $\widehat{A(U)}=q_1$ to match the previous notation. Because of the corkscrew condition, we get a uniform lower bound for the distance $|A(U)-\widehat{A(U)}|$ in terms of $l(Q)$ and $l(U)$, which thus suffices for our purposes.

Consider  the dyadic decomposition of $U$ by the David-Semmes cubes of radius comparable to $2^j$, attained from \cref{cubelemma} (3), for some sufficiently large negative integer $j\in \Z_{-}$ to be chosen later.

Consider the value,
\begin{align}\label{localdensity}
    \Theta(U):=\Theta_{\omega^x}(U)=\frac{c c_{U} }{( l(U))^{s}}\geq \frac{c_0}{(l(U))^{s}}.
\end{align}

We assume that for each dyadic cube $C\in \D_j(U)$ which is constructed above, we have,
\begin{align}\label{a}
  (a): \  \Theta(C):=\Theta_{\omega^x}(C)\in [M^{-1}\Theta(U), M\Theta(U)]
\end{align}


Further, we assume that for each Euclidean cube centered on $q_1$ with radius $r$, that
\begin{align}\label{b}
(b:\ \ \Theta(C(q_1,r))\leq M\Theta(U)
\end{align}

If this is not true for some Euclidean cube $B(q_1,r)$ with $r\leq l(U)$, and we have 

\begin{align}\label{contra}
\Theta(C(q_1,r))>M\Theta(U),
\end{align}
then the lemma is proved by an elementary covering of $B(q_1,r)$ by a set of dyadic cubes at the length scale $j$ with $2^{j} \leq r\leq 2^{j+1}$, since we get that for $M$ large enough dependent on $C,c_0$, one of these dyadic cubes must satisfy the condition of \cref{eq366} in which case the lemma is proved. 

More precisely, consider the covering of $C(q_1,r)$ by a set of disjoint dyadic cubes $\{Q_i:l(Q)\approx r\}$. By construction, the total diameter of the union of these disjoint dyadic cubes is bounded by a uniform multiple of $r$. If we have for each $i$, that, 
\begin{align*}
    \Theta(Q_i)<M_2 \Theta(U), 
\end{align*}
then we have, 
\begin{align*}
    \omega^x(B(q_1,r))\leq \omega^x(\cup_{i}Q_i)\leq M_2 \sigma(\cup_{i}Q_i)\Theta(U)\leq CM_2 \sigma(B(q_1,r))\Theta(U).
\end{align*}

So we have a contradiction to \cref{contra}, when choosing $M$ large enough compared to $C$ and choosing $M_2=\frac{M}{C}$.


On the other hand, we also cannot have $\Theta(B(q_1,r))\geq M\Theta(U)$ for some $r\geq 2l(U)$, since by a crude bound we must have that,
\begin{align*}
    \Theta(B(q_1,r))\leq c\frac{1}{(2 l(U))^s},
\end{align*}

and for $M$ large enough compared to $C, c_0$, this would give us a contradiction.


We write the contribution to the gradient at $q_1$, due to the harmonic measure in the annular region $A_m$ as $\nabla_{A_{m}}$ . We also write the contribution to the gradient at $q_1$, projected on to the unit outward normal $\hat{v_1}$, due to the annular region $A_m\cap H^{-}_{1}$, as $\proj  $, for each $m\geq 1$. 

From \cref{upperbound,lowerbound}, we have upper and lower bounds to the magnitude of this gradient. We consider the projection of this gradient along the outward normal to the plane $\Sigma_1$, and get with an extra $(\frac{1}{\sqrt{n}}\en)^w$ factor due to the projection, 
\begin{multline}\label{ineq}
          C_{3}M^{-1}\Theta(U)\beta_1^{w}\beta_1^{ws} (\kappa^m t_1)^{s-n+1}= C_{3}M^{-1}\Theta(U)\beta_1^{wn} (\kappa^m t_1)^{s-n+1}\\  \leq |\proj| \leq MC_{2}\Theta(U) (\kappa^{m} t_{1})^{s-n+1}.
\end{multline}
Note that the projection of this gradient, due to the harmonic measure in the ball $B(x_m, \frac{\beta_1^w\kappa^m t_1}{2})$, on to the direction $\hat{v}$, gives rise to an additional factor of $\beta_1^w$ in the left hand inequality in \cref{ineq} \footnote{Of course, if we have elements of the surface measure further away from the plane $\Sigma_1$, then we would have a factor bigger than $\beta_1^w$ for the projection. }. Note also that we may have several balls such as $B(x_m, \frac{\beta_1^w\kappa^m t_1}{2})$ , at least a distance $\beta_1^w \kappa^m t_1$ away from $\Sigma_1$, which gives rise to cancellations in the gradient. However, the gradient along the unit normal vector $\hat{v}$ only gets reinforced in such a case, and we have the lower bound attained in \cref{ineq}. 

Note that we may have a contribution to the gradient at $q_1$, due to the harmonic measure of
$\partial \Omega \cap A_m \cap H_1^+$.
By an identical argument as for the left hand side of \eqref{ineq}, we get that the magnitude of the gradient
due to this part, projected onto the inward normal $-\hat{\nu}$, is bounded from above by,
\begin{align}\label{eq43}
 C' \frac{t_1}{r_1} \Theta(U) (C_1 k^m t_1)^{s} \frac{1}{(k^m t_1)^{n-1}}
= \Theta(U) \frac{C t_1}{r_1} (k^m t_1)^{s-n+1},
\end{align}
for some uniform constant $C$.
Since by \eqref{separate}, we have
\begin{align}
\frac{t_1}{r_1} \ll \beta_1^{w(1+s)},
\end{align}
the contribution of this is negligible compared to the left hand side of \eqref{ineq}, and thus by slightly
altering the constant $C_3$ in \eqref{ineq}, depending on the constant $C$ in \eqref{separate}, we retain
the lower bound obtained in \eqref{separate}.
    \begin{itemize}
\item \textbf{The case $s=n-1$.} 





The magnitude of the component of the gradient at $q_1$ along $\hat v_i$, due to the first $N$ annular regions contained in
$B(q_1,t_1)\cap H^{-}_1$, i.e.\ $\{A_m \cap H^{-}_1\}_{m=0}^{N-1}$,
is bounded from above and below by
\[
C_3 M^{-1} N \Theta(U) \beta_1^{wn}
\;\le\;
|\sum_{m=0}^{N-1} \proj|
\;\le\;
M N \Theta(U) C_2 .
\]

Now consider the cube $C(q_1,\kappa^N t_1)$, and a subcube
$C(q_1,\gamma k^N t_1)$ for $\gamma=1/2$. For any $x \in C(q_1,\gamma k^N t_1)$ and
$y \in \Omega \setminus C(q_1,\kappa^N t_1)$, using \cref{mult expansion}, the double derivative of
$\mathcal{E}(x-y)$ in any coordinate direction, for fixed $y \in \Omega \setminus C(q_1,\kappa^N t_1)$, as a function of $x$,
is bounded from above by
\begin{align}\label{eq43}
\frac{c}{|x-y|^n}.
\end{align}

Note the standing assumptions in \cref{a,b}. So we have that in \cref{taylor} the second derivative contribution at $x$, for some $x \in C(q_1,\gamma \kappa^N t_1)$, in any coordinate direction, due to the harmonic measure in the region
$C(q_1,t_1) \setminus C(q_1,k^N t_1)$, is bounded in magnitude, up to a constant $C_4$, by,
\begin{align}\label{secondderivative}
\sum_{j=0}^{\infty} \frac{MC_3\Theta(U)}{( \kappa^N t_1) 2^j}
\;=\;
\frac{MC_4 \Theta(U))}{k^N t_1}.
\end{align}

By essentially the same argument, the second derivative contribution at $x\in C(q_1,\gamma \kappa^{N}t_1)$ due the harmonic measure in $\Omega\setminus C(q_1,t_1)$ is bounded from above by the same expression $\frac{MC_4 \Theta(U))}{ t_1}$.

We add to it the contribution to the second derivative, due to the pole at $x$, which also adds another term of the form $\frac{c_5 \Theta(U)}{\kappa^N t_1}$, noting that the distance from the pole $x$ to $B(q_1,\kappa^N t_1)$ is much larger than $\kappa^N t_1$, and thus we keep the right hand term with a slightly altered constant $c_4$ in \cref{secondderivative}. 




We choose below a point $q_2 \in B(q_1,\gamma \kappa^N t_1)$ that is some distance
$\theta_1 \gamma \kappa^N t_1$ from $q_1$, in the manner to be described now, with $1 > \theta_1 > \theta_0$, and
\begin{align}\label{eq45}
\theta_0 = \beta_1^{wn}.
\end{align}

We consider the plane passing through $q_1$ and perpendicular to the gradient due to the harmonic measure in $(C(q_1,t_1)\setminus C(q_1,\kappa^N t_1)) \cap \partial\Omega$,  which we write as $\nabla A_N$.
\begin{align}
\nabla A_N=  \sum_{m=0}^{N-1} \nabla A_m.
\end{align}
Call this plane $\Sigma_{2}$, and denote by $\hat u_{2}$ the unit normal to it
(we choose the direction $\hat u_{2}$ so that $\hat u_{2}\cdot \hat v > 0$).
By employing the argument of \cref{lemma6} exactly, we ensure that the point $q_2$ is at a
distance at least $\beta_1^{w}\gamma \kappa^N t_1$ from $\Sigma_{2}$, inside the cube $C(q_1,\gamma\kappa^N t_1)$.
 
Next, consider the gradient at $q_1$ due to the harmonic measure in $\Omega \setminus C(q_1, t_1)$, and the plane passing through $q_1$ perpendicular to this gradient, which we call $\Sigma_3$. If the point $q_2$ is a distance greater than $\frac{1}{2}\bigl(\beta_1^{2w}\gamma \kappa^N t_1\bigr)$ from $\Sigma_3$, then we are done. Otherwise, consider a ball $B(q_2, \frac{1}{2}\beta_1^{w}\gamma \kappa^N t_1)$ of radius $\frac{1}{2}\beta_1^{w}\gamma \kappa^N t_1$ centered on $q_2$. By repeating the argument of \cref{lemma6} now with respect to the plane $\Sigma_3$ restricted , we conclude that $q_2$ is a distance at least
\[
\frac{1}{2}\theta_0 \gamma \kappa^N t_1 =\frac{1}{2}\bigl(\beta_1^{wn}\gamma \kappa^N t_1\bigr)
\]
from $\Sigma_3$, inside the cube $C(q_1,\gamma\kappa^N t_1)$.

One sees that regardless of the orientation of the hyperplanes $\Sigma_1,\Sigma_2$ through $q_1$, we have now found a point $q_2 \in B(q_1,\gamma \kappa^N t_1)$ which is some distance
$\theta_1 \gamma \kappa^N t_1$ from $\Sigma_1,\Sigma_2$, with $1 > \theta_1 > \theta_0=\beta_{1}^{wn}$.

Now, the difference in the potential values between $q_1$ and $q_2$
due to the harmonic measure in $C(q_1,t_1)\setminus C(q_1,\kappa^N t_1)$ is given by using the Taylor theorem up to the second order; as in \cref{taylor}.

Note that due to the above considerations, the first term on the left in
\eqref{taylor} is at least
\[
\beta_1^{w(n+1)}\,\gamma\,k^{N }t_1\,C_3\,M^{-1}\,N\,\Theta(U).
\]
The second term on the right in \eqref{taylor} is bounded from above by
\begin{align}
M\gamma^2 \kappa^{2N} t_{1}^{2}\,\Theta(U)\frac{1}{\kappa^{N}t_1}
\;\lesssim\;
M\gamma^2 \kappa^{N} t_1\,\Theta(U).
\end{align}

This follows by truncating the sum in \cref{secondderivative} to only take into account the harmonic measure in the domain $C(q_1,t_1)\setminus C(q_1,\kappa^N t_1)$. 

Now, consider the difference in potential between the points $q_1$ and $q_2$ due to the harmonic measure in the domain $\Omega \setminus C(q_1, t_1)$, i.e. the contribution in $I_{\text{far}}$ in the expression in \cref{decomposition}.

We call the magnitude of the gradient at $q_1$, due to the harmonic measure in
$\Omega \setminus C(q_1, k^N t_1)$, by $T_N$. Note that the second derivative contribution from $I_{\text{far}}$ is bounded
crudely by the same expression as in \cref{secondderivative}.

We choose $N$ large enough so that,
\begin{align}\label{eq48}
\beta_1^{w(n+1)}\,\gamma\,k^{N }t_1\,C_3\,M^{-1}\,N\,\Theta(U)
\;\gg\;
M\gamma^2 k^{N }t_1\,\Theta(U).
\end{align}

Thus, we get
\[
\beta_1^{w(n+1)} C_3 N \;\gg\; M^2 \gamma\, C_5.
\]

It is enough to choose $\gamma=\frac{1}{2}$\footnote{Note that in the case of $s<n-1$, we will have to choose $\gamma$ adequately small, as will become clear later. In the present scenario with $s=n-1$, we only need to choose $\gamma=\frac{1}{2}$ so that there is enough separation between points in the ball $B(q_1, \gamma\kappa^N t_1)$ and $\Omega\setminus B(q_1, \kappa^N t_1)$ so that we can employ the estimates for the second derivative contribution coming from the annular region $B(q,\kappa^{N-1}t_1)\setminus B(q_1,\kappa^N t_1)$, using the Ahlfors-David regularity, as in the expression on the left hand side of \cref{secondderivative} .} and a positive integer $N$, so that,
\begin{align}\label{Nlowerbound}
N= \lceil C_7\,M^2\,\beta_1^{-w(n+1)}\, \rceil.
\end{align}

Thus we conclude that when $N$ is chosen as in \cref{Nlowerbound}, the second
derivative contribution to the difference in potential between $q_1$ and $q_2$,
due to $I_{\text{far}}, I_{\text{mid}}$, is negligible in comparison to the gradient contribution due to $I_{\text{mid}}$.

A priori for any $\kappa^{N_1}\ll \beta_1^{wn}$, the change in the gradient term of $I_{\text{mid}}$ in changing between two points $a,b\in C(q_1,\kappa^{N} t_1)$, with $|a-b|\leq \kappa^{N+N_1}\gamma t_1$, is bounded by the second derivative contribution in $I_{\text{mid}}$, and this is bounded from above by, 
\begin{align}\label{eq54}
    \frac{M\Theta(U)C_5}{\kappa^N t_1}\kappa^{N+N_1}t_1=M\Theta(U)C_5\kappa^{N_1},
\end{align}
since $|a-b|\leq \kappa^{N+N_1}t_1$. Also, the gradient term at $q_1$ is bounded from below by, 
\begin{align}\label{eq55}
    \beta_1^{wn}C_3 M^{-1}N\Theta(U).
\end{align}

Note that with the choice of $N$ above, we easily get by adjusting the constants, using the fact that $\kappa<1$, that,
\begin{align}\label{compp}
    \beta_1^{wn}C_3 M^{-1}N\Theta(U)\gg M\Theta(U)C_5 \gg \Theta(U)C_5\kappa^{N_1}.
\end{align}

Consider the dyadic cube of length scale $\kappa^{N_1}$ in which $q_1$ lies, and the center $c(q_1)$ of this dyadic cube. In this case, we have $|q_1-c(q_1)|\leq \kappa^{N+N_1}t_1\gamma$ . Similarly choose the center of the dyadic cube of length scale $\kappa^{N_1}$ in which $q_1$ lies, and the center $c(q_1)$ of this dyadic cube, and again we have, $|q_2-c(q_2)|\leq \kappa^{N+N_1}t_1\gamma$. 

 Using \cref{Nlowerbound}, we readily see that the contribution in \cref{eq55} is much larger than the contribution from \cref{eq54}, and by altering the constant $C_3$ we can keep the same expression for the gradient contribution from $I_{\text{mid}}$, at $c(q_1)$. We consider the planes $\Sigma_2,\Sigma_3$ to pass through the point $c(q_1)$. Without loss of generality, we can thus consider the points $c(q_1), c(q_2)$ to replace the points $q_1,q_2$ in the subsequent argument, and relabel them as $q_1,q_2$.

Apart from the dyadic $\mathcal{C}(q_1, \kappa^{N+N_1} t_1)$, we do a standard dyadic decomposition, to get that this $I_{\text{near}}$ contribution to the potential at $q_1$ is bounded from above by $M\thee C\kappa^{N}t_1.$
What remains, is to estimate the $I_{\text{near}}$ contribution to the potential at $q_1$ from the dyadic cube $\mathcal{C}(q_1, \kappa^{N+N_1} t_1)$.

Here we use the notation $\mathcal{C}(q_1,\kappa^{N+N_1}t_1)$ to mean the unique dyadic cube of length scale $\kappa^{N+N_1}t_1$, centered at $q_1$.

\begin{lemma}\label{lemma10}
There exists a point $q'_1\in C(q_1,\frac{1}{2}\kappa^{N+N_1}t_1)$, so that by changing $q_1$ to a point $q'_1 \in C(q_1,\kappa^{N+N_1}t_1)$, we get that the total contribution to the potential at $q'_1$ , due to
$C(q_1, \kappa^{N} t_1)$, is bounded from above by $C_8 M \kappa^{N} t_1 \, \Theta(U)$.
\end{lemma}

\begin{proof} Note that if for every
\begin{align}
q \in C\!\left(q_1, \tfrac12 \kappa^{N+N_1} t_1\right),
\end{align}
we have the contribution,
\begin{align}\label{local integral}
\int_{C(q_1, \kappa^{N+N_1} t_1)} \frac{d\omega(x)}{|x-q|^{n-2}}
\;\ge\;
C_9 M \kappa^{N} t_1 \, \Theta(U).
\end{align}

Then integrating in $q$, we get
\begin{align}\label{doubleintegral}
\int_{C\!\left(q_1, \tfrac12 \kappa^{N+N_1} t_1\right)}
\Bigg(\int_{C(q_1, \kappa^{N+N_1} t_1)}
\frac{\kappa(x)\,d\sigma(x)}{|x-q|^{n-2}}\Bigg)dq
\;\ge\;
C_9 M \Theta(U)\, \kappa^{N} t_1\,\cdot  \kappa^{N+N_1} t_1.
\end{align}

Here $d\omega(x)=\kappa(x)\,dx$, and $\kappa(x)$ is the Poisson kernel within this ball.
Interchanging the order of integration gives that the above integral
is bounded from above by
\begin{align}\label{eq53}
\int_{C(q_1, \kappa^{N+N_1} t_1)}
\kappa(x)\left(
\int_{C\!\left(q_1, \tfrac12 \kappa^{N+N_1} t_1\right)}
\frac{d\sigma(q)}{|x-q|^{n-2}}
\right)dx.
\end{align}

Here, the symbol $dq$ stands for the increment of the surface measure. Using the $(n-1)$ Ahlfors regularity and dyadic scaling of the integrand, the inner integral is bounded from above by
$C_{10}\, \kappa^{N+N_1} t_i$.

Also, by assumption, we have
\[
\Theta\left(C(q_1, \kappa^{N+N_1} t_1)\right)
=
\int_{C(q_1, \kappa^{N+N_1} t_1)} \kappa(x)\,d\sigma(x)
\;\le\;
M \left(\kappa^{N+N_1} t_1\right)^{n-1}.
\]

So we get an upper bound to \cref{eq53}, given by,
\[
M\, C_0\, \Theta(U)\, (\kappa^{N+N_1})^{\,n}.
\]
Whereas the lower bound to the left hand side of \cref{doubleintegral} is given by
\[
M\, \Theta(U)\, C_9\, \kappa^{N}\, (\kappa^{N+N_1})^{\,n-1}.
\]
Thus choosing \(N_1 \) large enough, uniformly dependent on \(M, , c_0, n\), we get a contradiction.
\end{proof}
Using \cref{mult expansion,local integral}, we see that by going from $q_1$ to a point $q'_1 \in C(q_1, \frac{1}{2}\kappa^{N+N_1}t_1)$ with the property that,
\begin{align}
    \int_{C(q_1, \kappa^{N+N_1} t_1)} \frac{d\omega(x)}{|x-q'_1|^{n-2}}
<
C_9 M \kappa^{N} t_1 \, \Theta(U),
\end{align}
the gradient contribution due to $I_{\text{mid}}$ at $q'_1$ changes by a negligible amount in comparison to the gradient contribution due to $I_{mid}$ at $q_1$, by the argument preceding the proof of \cref{lemma10}. Similarly, we choose a point $q'_2$ in the cube $C(q_2,\kappa^{N+N_1}t_1)$, so that,
\begin{align}
    \int_{C(q_2, \kappa^{N+N_1} t_1)} \frac{d\omega(x)}{|x-q'_2|^{n-2}}
<
C_9 M \kappa^{N} t_1 \, \Theta(U),
\end{align}

Note that the contribution to the difference in the potential between $q'_1,q'_2$ from $I_{\text{near}}$ is bounded from above crudely by $G_{1,\text{near}}+G_{2,\text{near}}$ where $G_{1,\text{near}}$(respectively $(G_{2,\text{near}}))$ is the net contribution to the potential due to $C(q_1,\kappa^{N}t_1)$ at $q_1$(respectively $q_2$.)

Using this, we have that the contribution to the difference in the potential between the points $q'_1, q'_2$, due to $I_{\text{near}}$, as well as the contribution due to the second derivative contribution from $\Omega\setminus B(q_1,\kappa^N t_1)$ are made negligible in comparison to the gradient contribution from $I_{\text{mid}}$, and thus the gradient contribution from $I_{\text{far}}$ has to balance this gradient contribution from $I_{\text{mid}}$.\footnote{Note that for $I_{\text{far}}$, we have no control over how the cancellations in the gradient contribution happens. } Thus, we are forced to have,
\begin{align}\label{eq60}
M N \Theta(U)\, C_2 \gamma \kappa^{N}t_1
\;\gtrsim\;
T_N \theta_1 \, \gamma\kappa^N t_1
\;\gtrsim\;
c\, \beta_1^{wn}\, \gamma\kappa^N t_1\, M^{-1} N \Theta(U) .
\end{align}

So,
\begin{align}\label{eq36}
T_N \;\le\; \frac{M N \Theta(U) C_2}{\theta_1}
\;\le\;
\left(\frac{1}{\beta_1}\right)^{wn} M N \Theta(U) C_2 .
\end{align}
Recall that we have chosen $\theta_1>\theta_0$, which was defined in \cref{eq45}.

Now we consider the number of iterations till we reach the cube
\(C(q_1, k^{N+N_2} t_1)\).
Here we have chosen \(N_2\) many further steps so that we have that the lower bound on the gradient term from $I_{\text{mid}}$ when considering the cube $C(q_1, \kappa^{N+N_2}t_1)$, is greater than the right hand side of \cref{eq36}.
\begin{align}\label{eq37}
c\, \beta_1^{w(n+1)}\, M^{-1}  \Theta(U) \,(N+N_2)
\;>\;
\left(\frac{1}{\beta_1}\right)^{wn} M N \Theta(U) C_2 .
\end{align}

So,
\[
(N+N_2) >\; C \left(\frac{1}{\beta_1}\right)^{w(n+3)} M^2 N .
\]

So,
\begin{align}\label{eq633}
N_2  >\; C \left(\frac{1}{\beta_1}\right)^{w(2n+1)}\cdot M^2 N - N
\;\simeq\;
C \left(\frac{1}{\beta_1}\right)^{w(2n+1)} M^2 N \end{align}

Thus, repeating the entire calculation within the cube
\(C(q_1, k^{N+N_2} t_1)\),
we get a contradiction since
\begin{align}\label{eq633}
T_{N+N_2}
\theta
\gamma\kappa^{N+N_2} t_1=T_{N}
\theta
\gamma\kappa^{N+N_2} t_1
>\;
c\, \beta_1^{w(n)}\, \gamma\kappa^{N+N_2} t_1 \, M^{-1} (N+N_2) \thee.
\end{align}

So, this means that, 
\begin{align}
   T_N\geq \frac{1}{\theta}  c\beta_1^{w(n)}M^{-1}\Theta(U) (N+N_2)\geq  c\beta_1^{w(n)}M^{-1}\Theta(U) (N+N_2), 
\end{align}

which combined with \cref{eq36,eq37}, gives us a contradiction.

Here again, within the Euclidean cube $C(q_1,\kappa^{N+N_2}t_1)$ we need to repeat the argument of \cref{lemma10} so that we consider all the dyadic cubes of length scale $\kappa^{N+N_2 +N_1}t_1$, to ensure that the potential contribution $I_{\text{1,near}}$ within the cube $C(q_1, \kappa^{N+N_2}t_1)$ as well as the second order contribution in \cref{taylor} from $I_{\text{mid}}, I_{\text{far}}$, are negligible in comparison to the gradient contribution from $I_{1,\text{mid}} $ which we now define to be the contribution from the region $C(q_1,t_1)\setminus C(q_1, \kappa^{N+N_2}t_1 )$.

Thus, the argument gives that there exists some dyadic cube contained in $Q$, of length scale at least $2^{-j_0}l(Q)$ for some $j_0 \geq 1$ where $l(U)$ is the length scale of $U$ so that its density is either greater than $M\Theta(U)$,\footnote{Recall that we have considered $q_1$ as the corkscrew point relative to the ball $\tfrac13 B_U$.} or smaller than $M^{-1}\Theta(U)$. In the former case, for $M$ sufficiently large, we also get that the density is greater than $MC\Theta(Q)$ and the argument is finished, leading to the use of \cref{lemma12}. In the latter case, we need the following lemma.

\begin{lemma}\label{lemma11} We can find some $M_1>1$ depending on $M$ and the ambient parameters
so that, if there is a cube $A_1\subset U$ with
\[
\frac{\omega(A_1)}{\sigma(A_1)} < M^{-1}\Theta(U),
\]
then there must exist a cube of the same generation, $A_2\subset U$, so that
\[
\frac{\omega(A_2)}{\sigma(A_2)} > M_1\Theta(U).
\]

Moreover, it suffices to take 
\begin{align*}
    M_1 =1+\frac{1}{2}\Big( \frac{1- M^{-1}\eta_2}{1-\eta_1} -1 \Big).
\end{align*}
\end{lemma}
\begin{proof}Consider that there is some $M_1 > 1$ so that,
\[
\frac{\omega(A_1)}{\sigma(A_1)} < M^{-1}\thee, 
\qquad
\frac{\omega(A_j)}{\sigma(A_j)} < M_1\thee,
\]
for all cubes $A_j$ belonging to the disjoint decomposition of $U$ at this dyadic scale, i.e $U= A_1\sqcup_{j\geq 2} A_j$.

Then, summing the contribution to the harmonic measure, we get
\begin{align*}
\omega(U)
&= \omega(A_1) + \omega(A_1^c) \\
&< M^{-1}\frac{\sigma(A_1)}{\sigma(U)} \, \omega(U)
  + M_1 \frac{\sigma(A_1^c)}{\sigma(U)} \, \omega(U) \\
&= \left(
    M^{-1}\frac{\sigma(A_1)}{\sigma(U)}
    + M_1\frac{\sigma(A_1^c)}{\sigma(U)}
  \right)\omega(U).
\end{align*}

Thus, if we have
\[
M^{-1}\frac{\sigma(A_1)}{\sigma(U)}
+ M_1\frac{\sigma(A_1^c)}{\sigma(U)}
< 1,
\]
we get a contradiction.

Now we have some parameters $\eta_1$ and $\eta_2$ dependent on $j_0$,with $\eta_2/\eta_1$ independent of $j_0$, and the parameters from \cref{cubelemma}, so that
\[
\eta_1 < \frac{\sigma(A_1)}{\sigma(U)} < \eta_2.
\]

So if we have
\[
M^{-1}\eta_2 + M_1(1-\eta_1) < 1,
\]
then we have a contradiction since
\[
M^{-1}\frac{\sigma(A_1)}{\sigma(U)} + M_1\frac{\sigma(A_1^c)}{\sigma(U)}
< M^{-1}\eta_2 + M_1(1-\eta_1).
\]

So to enforce the contradiction, we can work with any
\[
M_1 < \frac{1}{1-\eta_1}\bigl(1 - M^{-1}\eta_2\bigr).
\]

So if we require,
\begin{align}
\frac{1}{1-\eta_1}\bigl(1 - M^{-1}\eta_2\bigr) > 1. \\ \Leftrightarrow 1 - M^{-1}\eta_2 > 1 - \eta_1\\  \Leftrightarrow \eta_1 > M^{-1}\eta_2 \\ \Leftrightarrow  M > \frac{\eta_2}{\eta_1} > 1.
\end{align}
then we can choose some fixed $1<M_1<\frac{1}{1-\eta_1}\bigl(1 - M^{-1}\eta_2\bigr) $ and we are done.

In particular, it is enough to work with,
\begin{align}
    M_1 =1+\frac{1}{2}\Big( \frac{1- M^{-1}\eta_2}{1-\eta_1} -1 \Big).
\end{align}

Thus it is enough to start out by requiring our $M$ parameter to be bigger than $\eta_2/\eta_1$.
\end{proof}

Lastly, we also require $\kappa^{N+N_1}t_1$ to be small enough so that, for  each annular region $C(q_1,\kappa^m t_1)\setminus C(q_1,\kappa^{m+1}t_1)$, the ball $B(x_m, c\beta_{1}^{w}\kappa^{m+1}t_1)$ obtained from \cref{lemma6} contains an ample subset that consists of dyadic cubes contained entirely within $C(q_1,\kappa^m t_1)\setminus C(q_1,\kappa^{m+1}t_1)$. This ensures that we do not have any double counting of cubes intersecting the boundary of $C(q_1,\kappa^m t_1)\setminus C(q_1,\kappa^{m+1}t_1)$  , between adjacent annular regions. Thus it is enough to require for each $1\leq m\leq N$,
\begin{align}\label{eq70}
    \kappa^{N+N_1}t_1\leq c\beta_{1}^{w}\kappa^{N+1}t_1\leq c\beta_{1}^{w}\kappa^{m+1}t_1.
\end{align}

 It is enough to take, 
\begin{align}
    \kappa^{N_1 -1}\leq c\beta_{1}^{2w}.
\end{align}
This is used in the case where $s<n-1$ as well.
\bigskip
If we found a ball where the density is gained by at least a factor of $M\thee$, for $M$ large enough, then we are done. Note that by adjusting the constant pre-factor in the lower bound to $N_1$, we can ensure $N_1 \ll N_2$, which is what we require.

On the other hand, if we only find a dyadic subcube $D\subset U$ where the density drops by a factor at most $M^{-1}$, we now repeat the previous lemma to find another dyadic cube $D_1$ where the density increment occurs by a factor of at least $M_1>1$ constructed above, which means $\Theta D_1>M_{1}\Theta(U)$.

We then consider all the dyadic subcubes in $D_1$ and the corkscrew point relative to the ball $\tfrac13 B_{D_1}$, and repeat the argument of \cref{lemma9}. Note that in the process, considering the second order contribution from $I_{mid},I_{1,mid}, I_{far}$, we might have cubes whose density increases by a factor of at least $M\Theta(D_1)>MM_1\Theta(U)$, but again with a basic covering argument, and noting that $D_1\subset U$ and $\omega^x (U)\geq c_0$ with $x_0\in \tfrac13 B_U \setminus \partial\Omega$ and a crude bound of $1$ for the harmonic measure of any Euclidean cube, we get that such cubes necessarily lie within  $Q(\supset U)$.

In the process, we find either a dyadic cube of length scale at least $2^{-2j_0}l(Q)$ and whose density increases by a factor of $M\Theta(D_1)$ in which case we are done, or we find by use of lemma8 restricted to $D_1$, a dyadic subcube $D_2$ where the density increment is at least $M_1\Theta(D_1)\geq M_{1}^{2}\Theta(U)$. We next repeat the argument within the Euclidean cube $D_2$, and continue the process, and upon running this argument at most $O(k)$ many times where $M_{1}^{k}=M$, we find a subcube $Q_1\subset Q$  where the average density has increased by a factor of $M$ over the average density in $U$, and whose length scale is at least $2^{-Ckj_0}l(Q)=2^{-Cj_0\frac{\log M}{\log M_1}}l(Q)$, and we conclude the proof of \cref{lemma9} for the case $s=n-1$ by replacing $\Theta(U)$ with $\Theta(Q)$, up to a multiplicative factor. This concludes the proof of \cref{lemma9} for the case of $s=n-1$.

\bigskip
\item
      \textbf{The case $n-1-\delta_0 \leq s<n-1$.}

      The structure of the proof in this case is similar to the previous case of $s=n-1$. The new ingredient in this paper, is to give quantitative estimates on $\delta_0$ in terms of the parameters $\beta$ and the Ahlfors regularity parameters. For convenience, let us write $s=n-1-\delta$. 

      Without loss of generality, consider again the points  $p_1, q_1$, and the cubes $C(q_1,\kappa^m t_1)$ considered in the previous section. 
      

Starting with the cube $C(q_1,\kappa^N t_1)$, we consider the magnitude of the
gradient at $q_1$ along $\hat{\nu}$, due to the first $N$ annular regions that constitute $C(q_1,t_1)\setminus C(q_!,\kappa^{N}t_1)$. 

We we find a lower bound on the gradient at ,
projected along $\hat{\nu}$.
From \eqref{eq20}, we first get the expression
\begin{align*}
C_2\left(\frac{1}{\kappa^{n-1}}-1\right)\left( \frac{1}{\beta^{n-1}} - 1 \right)^{w}
\left( \beta^{w} \kappa^{m} t_1 \right)^{s}
&= C_2 (1-\beta^{n-1})^{w}
\beta_1^{w (s-n+1)} (\kappa^{m} t_1)^{s} 
\end{align*}

Note from \cref{kappa} that $\kappa$ is written as a function of $s$ and $C_1$, the Ahlfors-David parameter, and that we can bound $n-2<s\leq n-1$. We absorb the factor of $(1/\kappa^{n-1} -1)$ into the constant term. So, we get a contradiction to the Ahlfors regularity, as before, when we have
\begin{align}
C_3 (1-\beta_1^{n-1})^{w}
\frac{1}{\beta_1^{w\delta}} (\kappa^{m} t_1)^{s}
\le \frac{1}{2} C_1 (\kappa^{m} t_1)^{s}.
\end{align}

So,
\begin{align}\label{eq64}
(1-\beta_1^{n-1})
< \beta_1^{\delta} \left( \frac{1}{2} C_4^{-2} \right)^{\frac{1}{w}}.
\end{align}

So, this would force a ball at a distance at least
$\beta_1^{w} (\kappa^{m} t_1)$ away from $\Sigma_1$,
contained in $H^{-}_1$.
Thus we have a lower bound for the projection of the gradient due to the ball $B(q_1,t_1)\setminus B(q_1,\kappa^N t_1)$ onto the vector $\hat{\nu}$,
\begin{align}\label{eq70}
\gtrsim
\frac{M^{-1}\thee(\beta_1^{w} \kappa^{m} t_1)^{n-1-\delta}}{(\kappa^{m} t_1)^{n-1}}
\beta_1^{\omega}
= \frac{M^{-1}\thee\beta_1^{w(n-1-\delta)+\omega}}{(\kappa^{m} t_1)^{\delta}}
= \frac{M^{-1}\thee\beta_1^{w(n-\delta)}}{(\kappa^{m} t_1)^{\delta}} .
\end{align}

Choose $N$ so that $N\delta\log \frac{1}{\kappa}\approx \log 2$, so that $\kappa^{N\delta}=1/2$.

Calculating the lower bound to the gradient at \(q_1\),
due to \( C(q_1,t_1) \setminus C(q_1,\kappa^N t_1) \), gives us, upon summing
over the \(N\) many scales chosen above, we get
\begin{align}
|\sum_{m=0}^{N-1} \nabla A_m|
&\ge
\frac{M^{-1}\thee\beta_1^{w (n-\delta)}}{(\kappa^N t_1)^{\delta}}
\left(  \kappa^{\delta} + \kappa^{2\delta} + \cdots + \kappa^{N\delta} \right)
\\
&=
\frac{M^{-1}\thee\beta_1^{w (n-\delta)}}{(\kappa^N t_1)^{\delta}}
\frac{1-\kappa^{N\delta}}{1-\kappa^{\delta}}
\\
&\gtrsim
\frac{M^{-1}\thee\beta_1^{w (n-\delta)}}{2(\kappa^N t_1)^{\delta}}
\,
\frac{1}{\left( \frac{1}{\kappa} \right)^{\delta}-1}
\left( \frac{1}{\kappa} \right)^{\delta}.
\end{align}

Using the mean value theorem, we get that,
\begin{align}\label{mvt}
\Bigg(\Big( \frac{1}{\kappa} \Big)^{\delta}-1\Bigg)
&= \delta \Big(\frac{1}{\kappa}\Big)^{\delta_2} \log \frac{1}{\kappa},
\end{align}
for some $0\leq \delta_2\leq \delta$.

Hence
\begin{align}\label{eq822}
|\sum_{m=0}^{N-1} \nabla A_m|
\gtrsim
\frac{M^{-1}\thee\beta_1^{w (n-\delta)}}{(\kappa^N t_1)^{\delta}}
\left( \frac{1}{\kappa} \right)^{\delta-\delta_2}
\frac{1}{\delta \log \frac{1}{\kappa}}
=
\frac{M^{-1}\thee\beta_1^{w (n-\delta)}}{(\kappa^N t_1)^{\delta}\,
\kappa^{\delta-\delta_2}\,\delta \log \frac{1}{\kappa}} .
\end{align}
Using the analogous version of \cref{separate,eq43}, in this case, we easily see that the contribution from $H^{+}_1\cap \partial\Omega \cap (C(q_1,t_1)\setminus C(q_1,\kappa^N t_1))$ carries a factor of $t_1/r_1$ in place of the factor of $\beta_{1}^w$ in \cref{eq70} and is thus negligible in comparison to the contribution from \cref{eq70}.

Further, when restricted to \(C(q_1,\kappa^N t_1)\), and the contribution to the potential from $I_{\text{near}}$, we get that the contribution from all
the cubes apart from \(C(q_1,\kappa^{N+N_2} t_1)\), is bounded from above by,
\begin{align}
M\thee C (\kappa^N t_1)^{1-\delta}
\left( 1 + \kappa^{1-\delta} + \kappa^{2(1-\delta)} + \cdots \right)
&=
M\thee C (\kappa^N t_1)^{1-\delta}
\frac{1}{1-\kappa^{1-\delta}} .
\end{align}

Further, by an immediate application of the argument of  \cref{lemma10} in this case, by going down to a length scale $\kappa^{N+N_1} t_1$ we can choose a point \(q'_1\) in the cube \(C(q_1,\kappa^{N+N_1} t_1)\) so that the total contribution
to the potential at $q'_1$ is given by
\begin{align}\label{eq78}
\le 2M\thee C (\kappa^N t_1)^{1-\delta}
\left( \frac{1}{1-\kappa^{1-\delta}} \right).
\end{align}
Note that we can choose $N_2$ so that  \cref{compp} holds.
We introduced the \(\theta\) parameter in the previous section, so that \(\beta_1^{wn}\le \theta \le 1\). Considering the cube \(C(q_1,\kappa^N t_1)\) and choosing the point
\(q_2\) exactly as in the previous section within the cube $C(q_1,\gamma\kappa^{N}t_1$), so that the difference in the potential between the points
\(q_1\) and \(q_2\), due to the gradient from the harmonic measure contained in \(C(q_1,t_1)\setminus C(q_1,\kappa^N t_1)\),
is lower bounded by
\begin{align}\label{eq79}
\geq \frac{M^{-1}\thee\beta_1^{w (n-\delta)}}{(\kappa^N t_1)^{\delta}\,
\kappa^{\delta-\delta_2}\,\delta \log \frac{1}{\kappa}} \theta\gamma\kappa^N t_1
&\geq
\frac{M^{-1}\gamma\thee\beta_1^{w(2n-\delta)}}{\kappa^{\delta-\delta_2 }\log \frac{1}{\kappa}}
(\kappa^N t_1)^{1-\delta}\frac{1}{\delta}.
\end{align}

Disregarding the second order contribution due to the harmonic measure in
\(\Omega\setminus C(q_1,\kappa^N t_1)\) for the moment, to the potential difference between
\(q_1\) and \(q_2\), we get by comparing \cref{eq78,eq79} that when
\begin{align}\label{eq855}
\frac{\beta_1^{w(2n-\delta)}}{(\log \tfrac{1}{\kappa}) \delta}
\gg
C\left(\frac{1}{1-\kappa^{1-\delta}}\right)= C\Bigg( \frac{1}{(1-\delta)\log \frac{1}{\kappa}}\Big(\frac{1}{\kappa}\Big)^{\xi} \Bigg) ,
\end{align}
where $0<\xi<1-\delta$, and the last equality follows with a Taylor expansion, the contribution due to the harmonic measure in \(C(q_1,\kappa^N t_1)\) to this potential difference
is negligible. Note that \(\kappa\) is dependent on \(C_1\), 

Thus it is enough to require that
\begin{align}\label{eq85}
\frac{\beta_1^{w(2n-\delta)}}{\delta}(1-\delta)\geq \frac{\beta_1^{w(2n-\delta)}}{\delta_0} (1-\delta_0) \geq c(C_1,M)=c(C_1, c_0)\geq 3,
\end{align}

where we choose $\delta_0\leq 1/2$, and we choose $c(C_1,c_0)=\text{max}(3, C M^4 ((\log 1/\kappa)^2)$ is a constant dependent on $C_1$ and M, the Ahlfors-David regularity constant. Here, without loss of generality we can choose $c(C_1,c_0)$ to be greater than $3$.

Thus, we need to choose $\delta_0$ so that, 
\begin{align}
    \frac{\beta_1^{w(2n-\delta)}}{\delta_0}\gg c(C_1,c_0).
\end{align}

Note that previously in \cref{eq20,eq64}, we also required,
\begin{align}\label{eq82}
(1-\beta_1^{n-1}) < \beta_1^\delta \left(\frac{1}{2}C_{1}^{-2}\right)^{1/w}.
\end{align}

thus since $\delta_0\geq \delta$, it is enough to require that 
\begin{align}
    (1-\beta_1^{n-1}) < \beta_1^{\delta_0} \left(\frac{1}{2}C_{1}^{-2}\right)^{1/w}.
\end{align}

We claim that it is enough to take, for sufficiently large $D>1$
\begin{align}
w &= \frac{D}{\beta_1^n}, \ \delta_0 \leq \frac{1}{c(C_1,c_0)}\beta_1^{3nw}.
\end{align}

First we see that,
\begin{align}
    \delta_0 c(C_1,c_0)\leq \beta_{1}^{3nw}\ll \beta_{1}^{w(2n-\delta)}
\end{align}
since $\beta_1<1$ and $3n > (2n-\delta_0)$ when $n\geq 3$ which thus satisfies \cref{eq85}.

Next, we take the logarithm on both sides of \eqref{eq82}, 
using these values of $w$ and $\delta$, we obtain
\begin{align}
\frac{D}{\beta_1^n} \log(1-\beta_1^{n-1})
&< \delta w \log \beta_1 - \log 2 - 2\log C_1 .
\end{align}

Equivalently, we need to show
\begin{align}
-\frac{D}{\beta_1^n} \log(1-\beta_1^{n-1})
&> w\delta \log \frac{1}{\beta_1} + \log 2 + 2\log C_1 .
\end{align}

Since, $\delta\leq \delta_0$, it is thus enough to show that,
\begin{align}\label{eq93}
    \frac{D}{\beta_1}
&> w\delta_0 \log \frac{1}{\beta_1} + \log 2 + 2\log C_1.
\end{align}

For all $\beta_1<1$,  we have
\begin{align}
-\log(1-\beta_1^{n-1}) &\ge \beta_1^{n-1},
\end{align}
and also
\begin{align}
w\delta_0
&\leq  \frac{1}{c(C_1,c_0)}\frac{D}{\beta_1^n}\,\beta_1^{3nD/\beta_1^n} \notag\\
&=\frac{D}{c(C_1,c_0)} \,\beta_1^{\frac{3nD}{\beta_1^n}-n}.
\end{align}
We see that since we chose $c(C_1, c_0)\geq 3$, we have
\begin{align}\label{eq96}
   \frac{D}{3}\geq \frac{D}{3}\beta_{1}^{3nD/\beta_{1} -n}\geq \frac{D}{c(C_1,c_0)}\beta^{3nD/\beta_{1} -n}
\end{align}
Further, noting that $\beta<1$, we choose $D$ large enough so that
\begin{align}\label{eq97}
    D\geq  3 \ \text{max}(\log2,2\log C_1).
\end{align}
We note that,
\begin{align}  
\frac{D}{3\beta_1}> \frac{D}{c(C_1,c_0)}\beta_{1}^{\frac{3nD}{\beta^{n}_1}-n}\log \frac{1}{\beta_1}.
\end{align}
Thus using Eq , and the above we get that \cref{eq93} holds.

Thus, in particular the dimension drop argument works when
\begin{align}\label{impparameter}
\delta_0 =\text{min}\Big(\frac{1}{2} \frac{1}{c(C_1,c_0)}\beta_1^{\frac{3nD}{\beta_1^n}}\Big).
\end{align}

It remains to show that the contribution due to the second order term
in the Taylor expansion, from $\Omega\setminus C(q_1, t_1)$, and from $C(q_1,t_1)\setminus C(q_1,\kappa^{N}t_1)$ 
is negligible in comparison to the gradient contribution from $I_{\text{mid}}$.
We see that either of these contributions is bounded from above by
\begin{align}
\sum_{j=0}^\infty \frac{M C_4 \thee}{(2^j k^N t_1)^{1+\delta}}
&\le
\frac{M C_4 \thee}{(k^N t_1)^{1+\delta}}.
\end{align}

Thus the contribution of the second order term is bounded from above by
\begin{align}
\gamma^2 (k^N t_1)^2 \cdot \frac{M C_4 \thee}{(k^N t_1)^{1+\delta}}
&= \gamma^2 M C_4 \thee (k^N t_1)^{1-\delta}.
\end{align}

We recall that the gradient contribution from $I_{mid}$  is bounded from below by,
\begin{align}\label{eq1044}
\frac{M^{-1}\,\thee\,\beta_1^{w(2n-\delta)}}{\kappa^{\delta-\delta_2}\,\log\!\left(\tfrac{1}{\kappa}\right)\,\delta}
\,( \kappa^{N} t_1 )^{1-\delta_1}.
\end{align}

Thus, if
\begin{align}\label{eq955}
\frac{M^{-1}\gamma \,\thee\,\beta_1^{\omega(2n-\delta)}}{\kappa^{\delta}\,\log\!\left(\tfrac{1}{\kappa}\right)\,\delta}
\gg \gamma^{2} M C_4 \thee
\end{align}
then we are done.

Thus to satisfy \cref{eq955}, it is enough to take,
\begin{align}
\frac{M^{-1}}{\kappa^{\delta}\,\log\!\left(\tfrac{1}{\kappa}\right)}
\gg \frac{1}{c(C_1,c_0)}\gamma M C_4,
\end{align}
since we have already also considered earlier the condition,
\begin{align}
\frac{\beta_1^{w(2n-\delta)}}{\delta} \gg c(C_1,c_0).
\end{align}

Thus, it is enough to consider a value of $\gamma$ small enough
so that
\begin{align}\label{eq97}
\gamma \ll \frac{M^{-2}c(C_1,c_0)}{C_4\,\kappa^{\delta}\,\log\!\left(\tfrac{1}{\kappa}\right)}.
\end{align}

Recall that $\kappa$ is explicitly taken as
\begin{align}
\kappa = C_1^{-2/(n-1-\delta)}.
\end{align}

Note that since the second order contribution from $I_{\text{mid}}$ is negligible in comparison to the gradient contribution from $I_{mid}$, this enables us, as in the previous section, to a-priori consider the points $q_1,q_2$ to belong to the centers of respective dyadic cubes, and so the argument analogous to \cref{lemma10} also goes through exactly, for this case.

Next, for the upper bound on the gradient contribution from $I_{\text{mid}}$ to the difference in the potential
between $q_1$ and $q_2$, we have,
\begin{align}
&\leq
\frac{C\,\Theta(U)\,M\,\gamma\,\kappa^{N} t_1}{(\kappa^{N} t_1)^{\delta}}
\left(1+\kappa^{\delta}+\kappa^{2\delta}+\cdots\right) \\
&= C_1\,\Theta(U)\,M\,\gamma\,(\kappa^{N} t_1)^{1-\delta}
\frac{1}{1-\kappa^{\delta}} \\
&= C_1\,\Theta(U)\,M\,\gamma\,(\kappa^{N} t_1)^{1-\delta}
\frac{1}{\kappa^{\delta}}
\frac{1}{(\tfrac{1}{\kappa})^{\delta}-1} \\
&\leq
C\,\Theta(U)\,M\,\gamma\,(\kappa^{N} t_1)^{1-\delta}
\frac{1}{\kappa^{\delta-\delta_2}}
\frac{1}{\delta\,\log(\tfrac{1}{\kappa})},
\end{align}
for $\delta$ small enough, and the $\delta_2$ parameter coming from \cref{mvt} as before.

Thus, similar to \cref{eq60} for the case of $s=n-1$, here we have, in the cube \( C(q, \kappa^{N} t_1) \), using \cref{eq1044} for the left hand inequality,
\begin{align}\label{eq1044}
 \frac{M^{-1}\gamma\thee\beta_1^{w(2n-\delta)}}{\kappa^{\delta-\delta_2 }\log \frac{1}{\kappa}}
(\kappa^N t_1)^{1-\delta}\frac{1}{\delta}\leq T_N \theta \gamma \kappa^{N}t_1  \leq C\,\Theta(U)\,M\,\gamma\,(\kappa^{N} t_1)^{1-\delta}
\frac{1}{\kappa^{\delta-\delta_2}}
\frac{1}{\delta\,\log(\tfrac{1}{\kappa})}
\end{align}
Recall that we have, $\beta_1^{wn}\le \theta \le 1$.

In particular, we have from here, that,
\begin{align}
    T_N\leq\frac{1}{\theta}C\,\Theta(U)\,M\,\gamma\,(\kappa^{N} t_1)^{-\delta}
\frac{1}{\kappa^{\delta-\delta_2}}
\frac{1}{\delta\,\log(\tfrac{1}{\kappa})} \leq  \Big(\frac{1}{\beta_{1}^{wn}}\Big)C\,\Theta(U)\,M\,\gamma\,(\kappa^{N} t_1)^{-\delta}
\frac{1}{\kappa^{\delta-\delta_2}}
\frac{1}{\delta\,\log(\tfrac{1}{\kappa})}
\end{align}


Thus, if we further go down to a radial length \( \kappa^{N+N_2} t_1 \), obtain the same bounds on the second derivative contributions from $I_{\text{mid}}, I_{\text{far}}$, the potential contribution from $I_{\text{near}}$, by going down to a length scale of $\kappa^{N+N_2 }$, so that only the gradient contributions from $I_{\text{mid}}$ and $I_{\text{far}}$ dominate these previous contributions,  we require that
\begin{align}\label{eq104}
\frac{\beta_1^{\,w(2n-\delta)} M^{-1}\thee \gamma}
{\, \delta\kappa^{\delta-\delta_2} \log \frac{1}{\kappa}\,}
\left( \kappa^{N+N_2} t_1 \right)^{-\delta}
\;\geq
M\cdot \frac{1}{\beta_{1}^{wn}} \frac{M\thee\gamma }{\delta \kappa^{\delta-\delta_2} \log \frac{1}{\kappa}\, }
\left( \kappa^{N} t_1 \right)^{-\delta},
\end{align}
then making the same argument within the cube
\( C(q_1, \kappa^{N+N_2} t_1) \),
we obtain a contradiction.\footnote{Note that $\gamma$ is chosen to satisfy \cref{eq97}}

This is because within the cube $C(q_1,\kappa^{N+N_2}t_1)$, we must have analogous to \cref{eq633}, and using \cref{eq1055}, that, 
\begin{multline}
     \frac{M^{-1}\gamma\thee\beta_1^{w(2n-\delta)}}{\kappa^{\delta-\delta_2 }\log \frac{1}{\kappa}}
(\kappa^{N+N_2} t_1)^{1-\delta}\frac{1}{\delta}\leq T_{N+N_2} \theta \gamma \kappa^{N+N_2}t_1  \\ \Rightarrow T_{N+N_2}=T_N\geq \frac{M^{-1}\gamma\thee\beta_1^{w(2n-\delta)}}{\delta\kappa^{\delta-\delta_2 }\log \frac{1}{\kappa}}
(\kappa^{N+N_2} t_1)^{-\delta}
\end{multline}

To satisfy \cref{eq104}, it is enough to require that,
\[
\beta_1^{\,w(2n+2-\delta)} M^{-1} \kappa^{-N_2 \delta}
\;\gtrsim\; M^2,
\]
so
\[
\left( \frac{1}{\kappa} \right)^{N_2 \delta}
\;\gtrsim\;
M^3 \left( \frac{1}{\beta_1} \right)^{w(n+4-\delta)}.
\]

Hence,
\[
N_2 \delta \log \frac{1}{\kappa}
\;\gtrsim\;
3\log M + w(n+4-\delta)\log \frac{1}{\beta_1}.
\]

Therefore, when
\begin{align}\label{eq1188}
N_2
\;\gtrsim\;
\frac{
3\log M
+ \bigl( \frac{\log C}{\beta_1^n} \bigr)(n+4-\delta)\log \frac{1}{\beta_1}
}{
\delta \log \frac{1}{\kappa}
}.
\end{align}

 we have the required contradiction.

Finally we show that in switching from the point $q_1$ to the point $q'_1$, the contribution from the gradient term still remains dominant, when $|a-b|\leq \kappa^{N+N_1}\gamma t_1$. This is similar to the argument in the previous case of $s=n-1$. 

The change to the gradient contribution from $I_{mid}$, in moving from $q_1$ to $q'_1$  is bunded from above, by considering a Taylor expansion up to the first order, 

\begin{align}
    (\kappa^{N+N_1} t_1)\cdot  \frac{M C_4 \thee}{(k^N t_1)^{1+\delta}}=\kappa^{N_1}M C_4 \thee(\kappa^{N_1}t_1)^{-\delta}.
\end{align}
On the other hand, from \cref{eq822} we get that the magnitude of the gradient at $q_1$ is lower bounded by 
\begin{align}
\frac{M^{-1}\thee\beta_1^{w (n-\delta)}}{\,
\kappa^{\delta-\delta_2}\,\delta \log \frac{1}{\kappa}} (\kappa^N t_1)^{-\delta}
\end{align}
Thus we need to ensure that,
\begin{align}
    \frac{M^{-1}\beta_1^{w (n-\delta)}}{\,
\kappa^{\delta-\delta_2}\,\delta \log \frac{1}{\kappa}} \gg \kappa^{N_1}M C_4 \Leftrightarrow \kappa^{N_1}\ll    C_{4}^{-1} \frac{M^{-2}\beta_1^{w (n-\delta)}}{\,
\kappa^{\delta-\delta_2}\,\delta \log \frac{1}{\kappa}}
\end{align}
Thus we have,
\begin{align}
    N_1\gg \frac{\log (\frac{M^{-2}\beta_1^{w (n-\delta)}}{\,
\kappa^{\delta-\delta_2}\,\delta \log \frac{1}{\kappa}}) }{\log \frac{1}{\kappa}}
\end{align}
We choose the lowest possible value of $N_1$ satisfying above. Further, we can choose a value of $N_2$ from \cref{eq1188} so that $N_2 \gg N_1$.

Lastly, by an argument identical to the one for the case of (1) with $s=n-1$, we also iterate this argument $\sim \log M/\log M_1$ many times within the dyadic cube $U$, to conclude \cref{lemma9}.
\end{itemize}
\end{proof}

\subsubsection{Change of poles argument.}

Now we adopt the proof of Lemma 4.3 of \cite{Az20}, with modifications, to complete the argument.

\begin{lemma}\label{lemma12}
    Let $Q$ be a dyadic cube centered on $\partial \Omega$ and let
$p \in \Omega \setminus Q$, and the boundary $\partial\Omega$ is $(s,C_1)$- Ahlfors-David regular, with $n-1-\delta_0\leq s\leq n-1$. Consider the center $c(Q)$ of the dyadic cube $Q$, and as in \cref{lemma9}, consider the dyadic cube $U$ containing $c(Q)$ so that $l(U)=\frac{1}{C'}l(Q)$, $C'>1$, with $C'$ large enough depending on the parameters of \cref{cubelemma}(3).


Set $\omega = \omega^{p}_{\Omega}$.
Then there exists $\eta > 0$, and a dyadic cube
$Q_1 \subset U$ such that
$l(Q_1) \geq \eta l(Q)$ and
\[
\Theta^s_\omega(Q_1)
\notin
\bigl[
\frac{1}{2}\, \Theta^s_\omega(Q),
\,
2 \, \Theta^s_\omega(Q)
\bigr].
\]
\end{lemma}

\begin{proof} Consider the ball $B_0 = \frac{1}{5}B_U$ and the pole $p \in \Omega \setminus aB_0$,
with $a = 10 b^{-1} > b^{-1}$, and $b$ as in the Bourgain lemma.


 We assume that
\begin{equation}\label{eq91}
\omega^p\!\left(\tfrac{1}{10} B_0\right) \ge M_2^{-1} \, \omega^p(Q),
\end{equation}
with $M_2$ to be chosen momentarily. Note that the surfaces measures of $\frac{1}{10}B_0$ and $Q$ are comparable.

Assuming that \cref{eq91} does not hold, we consider the minimal disjoint covering $\mathcal{C}_1$ of $\frac{1}{10C'}B_0$ by dyadic cubes of diameter comparable to $l(\frac{1}{10C'}B_0)$. In the process, the union of all these dyadic cubes of  $\mathcal{C}_1$, is contained in $\frac{1}{10} B_0$. If the Poisson kernel of each of the dyadic cubes in this cover $\mathcal{C}_1$ is greater than $\frac{1}{2} \Theta^s_\omega(Q)$, we then easily get that,
\begin{align}
    \omega^{p}(\frac{1}{10}B_0)>\omega^{p}(\cup_{S\in \mathcal{C}_1} S) > \frac{1}{2}\sigma(\cup_{S\in \mathcal{C}_1} S) \Theta^s_{\omega^{p}}(Q)>\frac{1}{2}C_3\sigma(\frac{1}{10}B_0) \Theta^s_{\omega^{p}}(Q)=\frac{1}{2}C_4 \omega^{p}(Q).
\end{align}

Thus if we choose $\frac{1}{2} C_4>M^{-1}_2$, we get a contradiction. Thus it is enough to choose \begin{align}\label{eq1188}M_2 =3/C_4.\end{align}
Thus, at least one of the dyadic cubes in $\mathcal{C}_1$ must have Poisson kernel less than or equal to $M^{-1}_1 \Theta^s_\omega(Q)$ and we would be done. 

Thus assume that \cref{eq91} holds, and consider,
\[
E = \{ x \in \partial B_0 \cap \Omega : \delta_\Omega(x) > \varepsilon \},
\]
and let $\widetilde{\omega} = \omega_{\Omega\setminus B_0}$.

By the Markov property for harmonic measure, and the
Bourgain lemma, we get, for $\varepsilon > 0$ small enough,
\begin{align*}
\omega^p\!\left(\tfrac{1}{10} B_0\right)
&= \int_{\partial B_0 \cap \Omega} \omega^x\!\left(\tfrac{1}{10} B_0\right)
\, d\widetilde{\omega}^p(x) \\
&\le C \varepsilon^{\alpha} \, \widetilde{\omega}^p(\partial B_0 \cap \Omega)
   + \int_E \omega^x\!\left(\tfrac{1}{10} B_0\right) \, d\widetilde{\omega}^p(x).
\end{align*}

By the Bourgain lemma, we have $\omega^x(Q) \gtrsim 1$ on
$\partial B_0 \cap \Omega$, and thus by the maximum principle we get
\[
\widetilde{\omega}^p(\partial B_0 \cap \Omega)
\le \omega^p(Q)
\le M_2 \, \omega^p\!\left(\tfrac{1}{10} B_0\right).
\]
The last inequality follows from \cref{eq91}.

Thus, with $\varepsilon > 0$ small enough depending on $M_1$ and $a$,
we get from the above two inequalities that
\begin{equation}\label{eq:2}
\omega^p\!\left(\tfrac{1}{10} B_0\right)
\le 2 \int_E \omega^x\!\left(\tfrac{1}{10} B_0\right)
\, d\widetilde{\omega}^p(x).
\end{equation}

Let $B_j$ be a covering of $E$ by bounded number of balls centered on $E$,
of radius $\varepsilon/4$, so that $2B_j \subset \Omega$.
We show there is $t > 0$ (depending on $\varepsilon$ and $M_1$), and there
is $j$ so that if $B' = B_j$, we get
\[
\widetilde{\omega}^p(B') \ge t \, \omega^p\!\left(\tfrac{1}{10} B_0\right).
\]

If not, then for each $j$ we have
\[
\widetilde{\omega}^p(B_j) < t \, \omega^p\!\left(\tfrac{1}{10} B_0\right).
\]
Using the fact that harmonic measure is at most $1$, we get
\begin{align*}
\omega^p\!\left(\tfrac{1}{10} B_0\right)
&\lesssim \sum_j \int_{B_j \cap \partial B_0}
\omega^x\!\left(\tfrac{1}{10} B_0\right) \, d\widetilde{\omega}^p(x) \\
&\le \sum_j 1 \cdot \widetilde{\omega}^p(B_j \cap \partial B_0)
\lesssim_\eps t \, \omega^p\!\left(\tfrac{1}{10} B_0\right),
\end{align*}
which is a contradiction for $t$ small enough. Here, the implied constant also depends on the Ahlfors-David regularity constants through \cref{eq1188}.

Thus we have a ball $B'$ centered on $\partial B_0$ with $2B' \subset \Omega$ and
\[
r_{B'} = \frac{\varepsilon}{2},
\]
and so that
\begin{equation}\label{eq94}
\widetilde{\omega^p}(B') \gtrsim_{\varepsilon,M_1} \omega^p\!\left(\tfrac{1}{10}B_0\right)
\gtrsim M_2^{-1}\,\omega^p(Q).
\end{equation}

Note that the last inequality follows from \cref{eq91}. 
Thus we have using Harnack inequality, for any $Q_\subset Q$, that
\begin{align}\label{eq95}
\omega^p(Q_1)
\ge \int_{B'\cap B_0} \omega^x(Q_1)\,d\widetilde{\omega^p}(x)
\gtrsim_{\varepsilon,\delta} M_{2}^{-1}\omega^p(Q)\,\omega^x(Q_1).
\end{align}
In the last step we have used  \cref{eq94}.

We then get from \cref{lemma9}, for any $M$ large enough, the existence of a dyadic subcube $Q_1\subset Q$, so that,
\[
\omega^x(Q_1) 
> M\,\frac{\sigma(Q_1)}{\sigma(Q)} \omega^x(Q)\geq Mc_0 \frac{\sigma(Q_1)}{\sigma(Q)}.
\]

In the last step above, we have again used the Bourgain estimate. So, from \cref{eq95} above, we get, 
\[
\omega^p(Q_1)
\gtrsim M_{2}^{-1} Mc_0\omega^p (Q)\frac{\sigma(Q_1)}{\sigma(Q)}.
\]

Thus, if we choose $M$ so that $\frac{C_4}{3}Mc_0=3$, which gives us
\begin{align}\label{eq122}
M=\frac{9}{C_4 c_0},
\end{align}
we get the desired result.

\end{proof}

This concludes the result and leads to the following \cref{lemma11} in Section 6.

\section{Dimension drop of the harmonic measure}

Recall that the dimension of a Borel measure $\mu$ in $\mathbb{R}^n$ is defined as follows:
\[
\dim(\mu)
=
\inf\bigl\{ \dim(G) : G \subset \mathbb{R}^n \text{ Borel},\ \mu(G^{c})=0 \bigr\}.
\]


\begin{lemma}\label{lemma11}
For $n \ge 3$, $s>0$, $C_1>1$, suppose there exists $M=M(n,s,C_1)>1$, so that the following holds. Let $\partial\Omega \subset \mathbb{R}^{n}$ be an
$(s,C_1)$--AD regular set with $n-1-\delta_0 \leq s\leq n-1$. Let $\omega:=\omega^{p_0}$ be the harmonic measure with pole $p_0$, supported on $\partial\Omega$ and
$\eta \in (0,1)$ such that, for any dyadic cube $C_1 \subset \partial\Omega$, $0<l(C_1) \le
\operatorname{diam}(\partial\Omega)$, there exists a dyadic subcube $C_2\subset C_1$ with
\[
  \l(C_2)\geq \eta \l(C_1),
\]
satisfying 
\begin{equation}
\Theta(C_2)
\ge
M\Theta(C_1)
\qquad\text{or}\qquad
\Theta(C_2)
\le
M^{-1}\Theta(C_1).
\end{equation}
Here we have used the notation, $\Theta=\Theta^{s}_\omega$ . Then $\dim \omega < s$.
\end{lemma}

Note that because of the $s-$Ahlfors David regularity hypothesis, we have that, $\sigma(C_2)\approx l(C_2)^s,$ $ \sigma(C_1)\approx l(C_1)^s$.

An earlier version of this lemma is stated in terms of balls, instead of dyadic cubes, (see Lemma 2.8 of \cite{Tol24}), and goes back to \cite{Bou87,Bat96}. Using the covering arguments used in the beginning of \cref{lemma12}, and also \cref{cubelemma}, one sees that the hypothesis of \cref{lemma11} implies that of \cite{Tol24}. \footnote{In fact the discerning reader will see that in several places in the proof, one could have worked with balls instead of dyadic cubes.}

This completes the proof of \cref{mainthm}.

\section{Future directions}

    If we consider the setting of \cite{Tol24}, where the boundary $\partial\Omega$ is contained in a co-dimension one hyperplane $\Gamma$, then one considers corkscrew balls contained within this hyperplane and the nearest `corner' point $q_1$ on $\partial\Omega\cap \Gamma$ to this corkscrew point $p_1$ in question. There it was shown that the dimension drop phenomenon works for boundaries of dimension at least $n- 3/2 -\eps$ for some small enough epsilon. In future course of study, we hope to use the methods of this paper to study this question.


\section{Acknowledgements} The author is grateful to Stephen Montgomery-Smith and Ridhhipratim Basu for helpful discussions. 

\section{Use of Large Language Models} The figure and notation table in Section 4 was compiled with the assistance of the large language model Claude Sonnet 5.5 and reviewed by the author for accuracy. It also played a part in revising the paper.  No Large language model did not contribute to any of the mathematical content of this paper.

\bibliographystyle{alpha}
\newcommand{\etalchar}[1]{$^{#1}$}
\def\cprime{$'$}

\bigskip

Aritro Pathak, Chern Institute of Mathematics, Nankai University.

Email: ap7mx@missouri.edu

\end{document}